\documentclass{article}
\usepackage[utf8]{inputenc}
\usepackage[T1]{fontenc}
\usepackage{authblk}
\usepackage{amsmath,amscd}
\usepackage{amsfonts}
\usepackage{amssymb}
\usepackage{amsthm}
\usepackage{amsfonts}
\usepackage{graphicx}
\usepackage{subfigure}
\usepackage{xcolor}
\usepackage{verbatim}
\usepackage{enumerate}
\usepackage{pgfplots}   
\pgfplotsset{width=7cm,compat=1.13}
\newif\ifdraft
\long\def\comment#1{} \long\def\old#1{}

\numberwithin{equation}{section}
\numberwithin{figure}{section}
\newtheorem{theorem}{Theorem}[section]
\newtheorem{corollary}[theorem]{Corollary}
\newtheorem{lemma}[theorem]{Lemma}
\newtheorem{proposition}[theorem]{Proposition}

\newtheorem{definition}[theorem]{Definition}
\newtheorem*{thm}{Theorem}
\theoremstyle{remark}\newtheorem{remark}[theorem]{Remark}

\let\qqed=\qed
\def\QED{\qqed\medskip}
\let\qed=\QED

\title{Local Analysis of Loewner Equation}
\author[1,2]{Henshui Zhang}
\author[2]{Michel Zinsmeister}

\affil[1]{School of Mathematical Sciences, Fudan University, 200433 Shanghai, PR China.}

\affil[2]{IDP, Universit\'e d' Orl\'eans, 45067 Orl\'eans  Cedex 2, France.}
\begin{document}
\date{}
\maketitle

\begin{abstract}
    Let $\lambda:[0,+\infty)\mapsto\mathbb{R}$ be the driving function of a chordal Loewner process. In this paper we find new conditions on $\lambda$ which imply that the process is generated by a simple curve. This result improves former one by Lind ,Marshall and Rhode, and it particular gives new results about the case $\lambda(t)=cW_b(t)$, $W_b$ being a H\"{o}lder-$1/2$ Weierstrass function.
    In the second part we find new conditions on $\lambda$ implying that the process is generated by a curve. The main tool here is a duality relation between the real part and the imaginary part of the Loewner equation.
\end{abstract}

{\small{\bf Key words and phrases} \,\, Loewner equation, Boundary Behavior}

{\small{\bf 2010 Mathematics Subject Classification} \,\,30C20 ,\ 30E25,\ 30C62.} \vskip1cm

\section{Introduction}
\subsection{Background}
This paper concerns Loewner theory of planar growth processes. Before stating the results, and in order to put them in perspective, we begin by recalling the main features of this theory.

We will focus on one side of the theory, namely the chordal Loewner equation. Strictly speaking Loewner original theory is a variant, nowadays called radial Loewner equation, that he introduced  in order to solve the $n=3$ case of Bieberbach conjecture in 1923.

Let $\mathbb{H}$ be the upper half-plane $\mathbb{H}=\{z=x+iy\in\mathbb{C};\,y>0\}$.

Assume first that $\gamma:[0,T]\rightarrow\overline{\mathbb{H}}$ is continuous and injective with $\gamma(t)\in\mathbb{H},t\in(0,T],\gamma(0)=0$. We write $K_t=\gamma[0,t]$ and $H_t=\mathbb{H}\backslash K_t$. This growing closed set $K_t$ is called the hull of  the Loewner process. From Riemann mapping theorem, there is an unique conformal mapping $g_t:H_t\rightarrow\mathbb{H}$ satisfying the following expansion at $\infty$:  
\begin{equation}\label{hcap} g_t(z)=z+\frac{c(t)}{z}+O(\frac{1}{z^2}) , \end{equation}
with $ c(t)\in\mathbb{R}_+$, and $g_t$ can be extended by Schwarz reflection principle to $\mathbb{C}\backslash K_t\cup s(K_t)$ holomorphically, where $s(z)=\bar z$.

Since the function $t\mapsto c(t)$, called half-plane capacity, is continuous and increasing from $0$ to $\infty$, we may reparametrize $\gamma(t)$ so that  $c(t)=2t$. It can be proved that the limit  $\lambda(t)=\displaystyle\lim_{z\in\mathbb{H}_t,z\rightarrow\gamma(t)}g_t(z)$ exists and lies in $\mathbb{R}$. Moreover $t\mapsto\lambda(t)$ is continuous and $t\mapsto g_t(z)$ satisfy the following so-called chordal Loewner equation:
\begin{equation}\label{LE}
\dot{g}_t(z)=\frac{2}{g_t(z)-\lambda(t)},\; g_0(z)=z,\, z\in\mathbb{H}.
\end{equation}
The proofs of these facts can be found in \cite{rohde2011basic}. The function  $\lambda$ is called the driving function of the process. Notice that $\lambda(0)=0$, since $g_0=\mathrm{id}$.


There is a converse to the preceeding considerations. Indeed, given a continuous function $\lambda:\mathbb{R}_+\to \mathbb{R}$, by Cauchy-Lipschitz theorem for ODE's, for an initial value $z_0\in\bar{\mathbb{H}}$, there exists a unique maximal solution of (\ref{LE}) defined on $[0,T_{z_0})\leq+\infty$ and  $\displaystyle\lim_{t\rightarrow T_{z_0}^-}g_t(z_0)=\lambda(T_{z_0})$ in the case $T_{z_0}<+\infty$. We say that $T_z$ is the capture time of $z$. We define for $t\geq 0$, 
$$K_t=\{z|T_z\leq t\},$$
 and it happens that $g_t$ is for all $t$ the Riemann mapping $H_t=\mathbb{H}\backslash K_t\rightarrow\mathbb{H}$ satifying the "hydrodynamic" normalization (\ref{hcap}).

For the process we started with, $K_t=\gamma([0,t])$ and we say that the process is generated by the (simple) curve $\gamma$. More generally we will say that the Loewner process driven by the function $\lambda$ is generated by the curve $\gamma$ if there exists a continuous function $\gamma:\mathbb{R}_+\to\overline{\mathbb{H}}$ with $\gamma(0)=0$, not necessarily injective, such that for all $t\geq 0$,
$H_t$ is the unbounded connected component of  $\mathbb{H}\backslash K_t$.

This is almost surely the case for the processes $SLE_\kappa$ which were introduced by Oded Schramm \cite{schramm2000scaling} and which are the chordal Loewner processes driven by $\lambda$ defined by
$$\lambda(t)=\sqrt{\kappa}B_t,$$
$B_t$ being a real standard Brownian motion. This is a deep theorem which is due to Rohde and Schramm \cite{rohde2011basic} for $\kappa\neq 8$ and to Lawler, Schramm and Werner \cite{lawler2011conformal} for $\kappa=8$.\\
Rohde and Schramm have precised this theorem by proving that $SLE_\kappa$ undergoes the following phases:
\begin{enumerate}[i.]
\item $0\leq\kappa\leq 4$, $\gamma$ is almost surely injective.
\item $4<\kappa<8$, $\gamma$ is almost surely a non-simple but nowhere-dense path.
\item $8\leq\kappa$, $\gamma(t)$ is almost surely a space-filling curve.
\end{enumerate}

Marshall and Rohde \cite{Marshall2005The} were the first authors to exhibit Loewner processes that are not generated by a curve and investigated sufficient conditions on the driving function that implies that the process is generated by a curve. Together with J.Lind \cite{lind2005sharp} they have shown that if the $1/2$-Lipschitz norm
$$\Vert \lambda\Vert_{1/2}=\sup{\frac{\vert\lambda(y)-\lambda(x)\vert}{\vert y-x\vert^{1/2}},\,x<y\in\mathbb{R}_+}$$
is $<4$ then the process is generated by a simple curve, and their example of a Loewner process not generated by a curve satisfies
$$ \Vert \lambda\Vert_{1/2}=4.$$

Later these three authors generalized this result in \cite{lind2010collisions}, in the following way:

We say that a function $\lambda:[0,T]\rightarrow\mathbb{R}$ is locally $\frac{1}{2}$-Lipschitz with norm $\leq C$ if there exists $\delta\in(0,1)$ such that $|\lambda(t)-\lambda(t')|<C\sqrt{t'-t}, \forall 0\leq t<t'<T$ with $|t-t'|<\delta(T-t)$.

\begin{thm}[LMR]
Let $\lambda$ be locally $\frac{1}{2}$-Lipschitz with norm $C<4$ on $[0,T]$ then
\begin{enumerate}[1)]
    \item If $\displaystyle\lim_{t\rightarrow T^-}\frac{\lambda(T)-\lambda(t)}{\sqrt{T-t}}=\kappa<4$ then the process is driven by a simple curve $\gamma:[0,T]\rightarrow\bar{\mathbb{H}}$ with $\gamma(T)\in\mathbb{H}$.
    \item If $\displaystyle\lim_{t\rightarrow T^-}\frac{\lambda(T)-\lambda(t)}{\sqrt{T-t}}=\kappa>4$ then the process is driven by a curve $\gamma:[0,T]\rightarrow\bar{\mathbb{H}}$ with $\gamma(T)\in\mathbb{R}$ or $\gamma(T)\in\gamma([0,T))$.
\end{enumerate}
\end{thm}

\subsection{Main Results}
In this paper we deal with two problems:
\begin{itemize}
\item $(P_1)$:  Given that a Loewner process is generated by a curve, what kind of condition on its driving function $\lambda$ will imply that this curve is simple?
\item $(P_2)$:  Given a Loewner process, what properties of its driving function imply that it is generated by a curve?
\end{itemize}
 The results of this paper are the content of two theorems. The first one is about $(P_1)$,  the other about $(P_2)$.

\begin{theorem}\label{THM1}
Let $\lambda:[0,T]\rightarrow\mathbb{R}$ be a continuous function such that the corresponding process is generated by a curve $\gamma$. For $t>0$, define
 $$ a(t)=\varliminf_{s\rightarrow t^-}\frac{|\lambda(t)-\lambda(s)|}{\sqrt{t-s}},b(t)=\varlimsup_{s\rightarrow t^-}\frac{|\lambda(t)-\lambda(s)|}{\sqrt{t-s}}.$$
If 
$$\forall t\in(0,T],\,a(t)<4\; and\; b(t)<\max\{(4,a(t)+\frac{4}{a(t)}\},$$ 
then the curve $\gamma$ is simple, and the result is sharp.
\end{theorem}

This theorem is a generalization of [LMR] where only the case $a(t)=b(t)<4$ is considered.

Before we state the second result, let us recall a well-known theorem (see \cite{lawler2008conformally}):
\begin{thm}[A]\label{Lawler}
A Loewner process $(g_t),t\in[0,T)$ is generated by a curve if and only if $\displaystyle\lim_{z\rightarrow\lambda(t)}g_t^{-1}(z)$ exists and defines a continuous function on $[0,T)$. 
\end{thm}
 By \cite{pommerenke2013boundary}, we would like to consider not only this limit but also the limit along curves which will be disscussed in the last section. When the limit does not exist we call a point $z$ a limit point of $g_t^{-1}$ if there exists $z_n\rightarrow \lambda(t)$ s.t. $\displaystyle\lim_{n\rightarrow\infty}g^{-1}_t(z_n)=z$. Such a limit point will be said to be accessible if there exists a curve $\beta:[0,1)\rightarrow\bar{\mathbb{H}}$ s.t. $\beta(0)=t$ and $\beta((0,1))\subset\mathbb{H}_t$.

Our second result is
\begin{theorem}\label{THM3}
Let $\lambda:[0,T]\mapsto\mathbb{R}$ be a continuous function such that the corresponding Loewner process is generated by a curve on $[0,T)$.

Let $\displaystyle a=\varliminf_{t\rightarrow T^-}\frac{\lambda(T)-\lambda(t)}{\sqrt{T-t}},b=\varlimsup_{t\rightarrow T^-}\frac{\lambda(T)-\lambda(t)}{\sqrt{T-t}}$. 

If $a\geq 5$ and $b<a+\frac 4a$, then the Loewner process is generated by a curve $\gamma$ in $[0,T]$ and $\gamma(T)\in\mathbb{R}$ or $\gamma(T)\in\gamma([0,T))$.
\end{theorem}
Since the local Lipschitz condition in [LMR] implies that the Loewner process is generated by a curve in $[0,T)$, again this theorem generalizes the second part of [LMR] which only covers the case $a=b\geq 5$.

\section{Transformation of the Real Loewner Equation}
\subsection{Local time change}
From \cite{lawler2008conformally}, we know that if the Loewner process $(g_t)$ is generated by a self -intersecting curve $\gamma$ with $\gamma(s_1)=\gamma(s_2),s_1<s_2$, then, if we choose $s_1<r<s_2$ and  consider the driving function $\lambda_r:\,t\mapsto\lambda(t+r)$, this Loewner process is generated by $\gamma_r:\,t\mapsto g_r(\gamma(r+t))$. We then have $\gamma_r(s_2-r)=g_r(\gamma(s_2))=g_r(\gamma(s_1))\in\mathbb{R}$. If $T_z^r$ denotes  the capture time of $z$ for the driving function $\lambda_r$, then the previous conclusion is equivalent to 
$$ T^r_{g_r(\gamma(s_1))}=s_2-r<+\infty .$$
 From this it becomes clear that the key point to prove that a Loewner process is  generated by a simple curve is to decide whether  $T^s_x<+\infty$ for all $s>0$ and $x\in\mathbb{R}$.

\begin{definition}
Let $\lambda(t)$ be a driving function,and $T$ a positive real. We say that $\lambda$ is captured at time $T$ if $\exists s>0, x\in\mathbb{R}$ s.t. $T^s_x=T-s$. We say that $\lambda$ is uncaptured if it is not captured at any time $T$.
\end{definition}

The preceeding discussion may be summarized by saying that whenever the process driven by $\lambda$ is known to be generated by a curve, then this curve is simple if and only if $\lambda$ is uncaptured. The main interest of this statement is that we need only study the Loewner equation on the real line.

Hence we consider the real Loewner equation:
\begin{equation}\label{RLE}
    \dot{X}(t)=\frac{2}{X(t)-\lambda(t)}, X(0)=X_0
\end{equation}
We assume wlog that $X_0>\lambda(0)$. Notice that we then have $X(t)>\lambda(t)$ for $t\in [0,T_{X_0})$; from this and (\ref{RLE}), we find that $X$ is actually increasing on $[0,T_{X_0})$.

If we have $\lim_{t\rightarrow1^-}X(t)=\lambda(1)\iff T_{X_0}=1$, that is if we are in the captured case (at time $T=1$,which we can assume wlog), set 
$$\lambda^-(t)=(\lambda(1)-\lambda(t))/\sqrt{1-t},X^-(t)=(\lambda(1)-X(t))/\sqrt{1-t}.$$
The equation (\ref{RLE}) becomes
$$\dot{X}^-(t)=\frac{1}{2(1-t)}\left(-\frac{4}{\lambda^-(t)-X^-(t)}+X^-(t)\right).$$
We now use a time change to get rid of the time term, namely $\sigma:[0,+\infty)\rightarrow[0,1),\,t\mapsto1-e^{-2t}$. Setting $x(t)=X^-(\sigma(t)),\xi(t)=\lambda^-(\sigma(t))$, we have
\begin{equation}\label{HRLE}
\displaystyle\dot{x}(t)=x(t)-\frac{4}{\xi(t)-x(t)},\;x(0)<\xi(0).
\end{equation}
This is the real Loewner equation in the  H\"{o}lder-$\frac{1}{2}$ point of view: the  functions $\xi$ and $x$ are continuous function in$[0,+\infty)$, and  $\xi$ is the new driving function. Since $X$ is increasing in $[0,1)$, we have $\lambda(1)=X(1)>X(t)>\lambda(t)$, from which we can draw two consequences:
\begin{enumerate}
\item Necessarily $\lambda(1)$ is a record, i.e. $\lambda(1)>\lambda(t),\,t\in[0,1),$
\item $\forall t>0,\;\xi(t)>x(t)>0.$
\end{enumerate} 

Let us assume conversely that the equation (\ref{HRLE}) has a positive solution defined on $[s,+\infty)$ with initial condition $x(s)<\xi(s)$, then it follows that the equation (\ref{RLE}) is vanishing at time $1$. Indeed, since $x(t)>0,\,t>0$, we have $X(t)<\lambda(1),\,t\in[0,1)$: but on the other hand $X(t)>\lambda(t),\,t\in[0,1)$ so that $X(1)=\lambda(1)$.\\
In this case, we say that $\xi$ and the equation (\ref{HRLE}) are captured,  the solution $x$ being then said to be captured . We only consider the captured time $1$ since when $\lim_{t\rightarrow T^-}X(t)=\lambda(T)$, we can do the same transformation and time change by only changing $1$ into $T$,  the equation happening to be the same.

\begin{remark}\label{solution}
When $\xi$ is a constant $c$, which corresponds to the case when the driving function is $\lambda(t)=c-c\sqrt{1-t}$, The equation is the linear ODE for the function $t(z)$
$$\mathrm{d}t=\frac{c-z}{cz-z^2-4}\mathrm{d}z,$$
this equation can be solved directly, and the solutions are different when $c>4,c=4$ and $0<c<4$. Tranforming it back, we obtain the solution of (\ref{RLE}).
\end{remark} 

\subsection{Real Loewner Equation}
In this subsection, we will analyze equation (\ref{HRLE}) and prove theorem \ref{THM1}. To this end we assume that $x$ is a captured solution of driving function $\xi$, and we construct a correspondence between $\xi$ and $x$. We define two sets of functions as:
\begin{align*}
\mathcal{L}&=\{\varphi\in L^1([0,+\infty),\;\varphi>0,\;\lim_{t\to\infty}e^{2t}\varphi(t)=+\infty\}\\
\mathcal{I}&=T(\mathcal{L}),
\end{align*}
where 
$$T(\varphi)(t)=e^t\int_t^\infty \varphi(s)\mathrm{d}s.$$

We then rewrite ~(\ref{RLE}) as 
\begin{equation}\label{RKEY}
    \xi(t)=x(t)+\frac{4}{x(t)-\dot{x}(t)},\,x(0)<\xi(0)
\end{equation}
and let us  define the nonlinear operator $F(\varphi)=\varphi+4/(\varphi-\dot{\varphi})$. We then have:
\begin{theorem}\label{RSTRU}
The equation ~(\ref{HRLE}) with driving function $\xi$ has a captured solution if and only if there is a $\Phi\in \mathcal{I}$, such that $\xi=F(\Phi)$.
\end{theorem}
\begin{proof}
Let us first assume that $x$ is a captured solution. Since $X$ and $\lambda$ are both continuous at $1$, it is easy to check that
$$\displaystyle \lim_{t\rightarrow+\infty}e^{-t}x(t)=\lim_{t\rightarrow 1}\sqrt{1-t}X^-(t)=0,$$ $$\displaystyle \lim_{t\rightarrow+\infty}e^{-t}\xi(t)=\displaystyle \lim_{t\to 1}\sqrt{1-\lambda}\lambda^-(t)=0.$$
Multiplying ~(\ref{RKEY}) by $e^{-t}$, and noticing that $\dot{x}(t)-x(t)<0$ since $x<\xi$, we have 
$$\displaystyle\lim_{t\rightarrow+\infty}\frac{4}{e^t(x(t)-\dot{x}(t))}=0,$$
and
$$\displaystyle\lim_{t\rightarrow+\infty}e^{2t}\frac{\mathrm{d}}{\mathrm{d}t}(e^{-t}x(t))=\lim_{t\rightarrow+\infty}e^t(\dot{x}(t)-x(t))=-\infty.$$
Setting $\varphi(t)=-\mathrm{d}(e^{-t}x(t))/\mathrm{d}t,$  we get
$$\displaystyle\int_0^{+\infty}\varphi(s)\mathrm{d}s<+\infty,\,\varphi(t)>0,\,\lim_{t\rightarrow+\infty}e^{2t}\varphi(t)=+\infty,$$
 and $\varphi$ is continuous, so that the function $\displaystyle \Phi:\,t\mapsto e^t\int_t^{+\infty}\varphi(s)\mathrm{d}s$ is an element of $\mathcal{I}$.

Conversely, if  $\Phi\in \mathcal{I}$ is such that $\xi=F(\Phi)$, then, by the time change, we have
$$\lambda(t)=\lambda(1)-\int_{-\frac{1}{2}\ln(1-t)}^{+\infty}\varphi(s)\mathrm{d}s-\frac{4(1-t)}{\varphi(-\frac{1}{2}\ln(1-t)),}$$
$$X(t)=\lambda(1)-\int_{-\frac{1}{2}\ln(1-t)}^{+\infty}\varphi(s)\mathrm{d}s,$$
and we easily check that $\lambda$ and $X$ satisfy the original Loewner equation (\ref{LE}) and that they are continuous and equal  at time $t=1$. 
\end{proof}
The set $\mathcal{I}$ is a cone, meaning  that if $\Phi_1, \Phi_2\in \mathcal{I}$ and $c>0$, then $\Phi_1+\Phi_2\in \mathcal{I}$ and $c\Phi_1\in \mathcal{I} $. The set $\mathcal{I}$ is  also a subspace of $C_1(\mathbb{R}_+)$. Lind has shown in \cite{lind2005sharp} that if $\forall t\geq 0,\xi(t)<4$, then $\xi\not\in F(\mathcal{I})$.

The next lemma, which is a generalization of Lind's theorem, is the key step in the proof of theorem \ref{THM1}: we define $$ a=\varliminf_{t\rightarrow T^-}\frac{\lambda(T)-\lambda(t)}{\sqrt{T-t}},\; b=\varlimsup_{t\rightarrow T^-}\frac{\lambda(T)-\lambda(t)}{\sqrt{T-t}}.$$

\begin{lemma}
If $\lambda$ is captured at time $T$, then $a$ and $b$ are both nonnegative or nonpositive and $\vert b\vert\geq \max(4,\vert a\vert+4/\vert a\vert)$ if $\vert a\vert < 4$.
\end{lemma}
\begin{proof}

Assume that $x$ is a captured solution with driving function $\xi$. From the equation [\ref{RLE}] we see that $X$ must be strictly monotone, implying that $a$ and $b$ are both either negative or positive. Changing $\lambda$ to $-\lambda$ if necessary, we assume from now on that $a\geq 0$. Using the time change, we get $\displaystyle a=\varliminf_{t\rightarrow\infty}\xi(t)$ and $\displaystyle b=\varlimsup_{t\rightarrow\infty}\xi(t).$\\
Suppose  that $0<a<4$: let $\varepsilon_0>0$ such that $a+\varepsilon_0<4.$ Because $\displaystyle\varliminf_{t\rightarrow\infty}\xi(t)=a$ there exist two sequences $t_n\uparrow+\infty$ and $\varepsilon_n\downarrow 0$ such that $\xi(t_n)<a+\varepsilon_n\leq a+\varepsilon_0<4,n\geq 1$. Since $x(t_n)+\frac{4}{x(t_n)}\geq4$, (\ref{RKEY}) together with the above imply that $\dot{x}(t_n)<0,n\geq 1$.

There are two cases:
\begin{enumerate}
\item $\exists n_0>0;\xi(t_{n_0})<a+\varepsilon_0<4$ and $\dot{x}(t)<0, \forall t\geq t_{n_0}$. Then the function $x$, being positive and decreasing on $(t_0,+\infty)$, must converge to a limit $c\geq 0$ as $t\rightarrow\infty$, which implies that $\displaystyle\varlimsup_{t\rightarrow\infty}\dot{x}(t)=0$. Thus there exists a sequence $s_n\uparrow+\infty$ such that $c\leq x(s_n)<c+\frac{1}{n}, 0>\dot{x}(s_n)>-\frac{1}{n}$. Putting this information in (\ref{RKEY}) we get 
$$\xi(s_n)=x(s_n)+\frac{4}{x(s_n)-\dot{x}(s_n)}\geq c+\frac{4}{c+\frac{2}{n}}$$
which implies that $b=+\infty$ if $c=0$ and $b\geq c+\frac{4}{c}$ otherwise. On the other hand :
$$c\leq x(t_n)\leq\xi(t_n)<a+\varepsilon,n\leq 1\Rightarrow c\leq a$$
Finally, we get $b\geq\min_{c\leq a}\{c+\frac{4}{c}\}=\max{(4,a+\frac 4a)}$.
\item If the first case does not occur we define  for all $ n\geq 1$  $T_n\geq t_n$ as the first time when $\dot{x}(T_n)=0$. Putting again this information in (\ref{RLE}), we get that 
$$\xi(T_n)=x(T_n)+\frac{4}{x(T_n)},\;x(T_n)\leq x(t_n)\leq \xi(T_n)\leq a+\varepsilon_n.$$

We conclude that $\xi(T_n)\geq\min_{d<a+\varepsilon_n}\{d+4/d\}=\max\{4,a+\varepsilon_n+\frac{4}{a+\varepsilon_n}\}$ and, as above, $\displaystyle b=\varlimsup_{t\rightarrow\infty}\xi(t)\geq\lim_{n\rightarrow\infty}\max\{4,a+\varepsilon_n+\frac{4}{a+\varepsilon_n}\}=\max\{4,a+\frac 4a\}$.
\end{enumerate}

To make sure that the bound is optimal, we give an example first in the case $a\leq 2$:

Set $\alpha_k=\frac{k(k+1)\pi}{2}+\pi-\frac{\pi}{k+1},\beta_k=\frac{k(k+1)\pi}{2}+\pi-\frac{\pi}{k+2},k=0,1,\ldots$
\begin{equation}
x(t)=
\begin{cases}
a+\frac{1}{\ln t}\cos({(k+1)(k+2)(t-\alpha_k)}) & t\in [\alpha_k,\beta_k)\\
a-\frac{1}{\ln t}\cos({\frac{t-\beta_k}{k+1}}) & t\in [\beta_k,\alpha_{k+1})).
\end{cases}
\end{equation}
It is not difficult to check that 
$$\varliminf_{t\rightarrow+\infty}\xi(t)=\lim_{n\rightarrow+\infty}\xi(\frac{\alpha_n+\beta_n}{2})=a,\varlimsup_{t\rightarrow+\infty}\xi(t)=
\lim_{n\rightarrow+\infty}\tilde{\xi}(\beta_n)=a+\frac{4}{a}.$$ Figure \ref{example2.6} shows the case $c=3/2$: the upper graph is that of $\xi$, the lower one of $x$, and the three lines are $y=0,c,c+4/c$.\\
\begin{figure}
\centering
\includegraphics[width=0.7\textwidth]{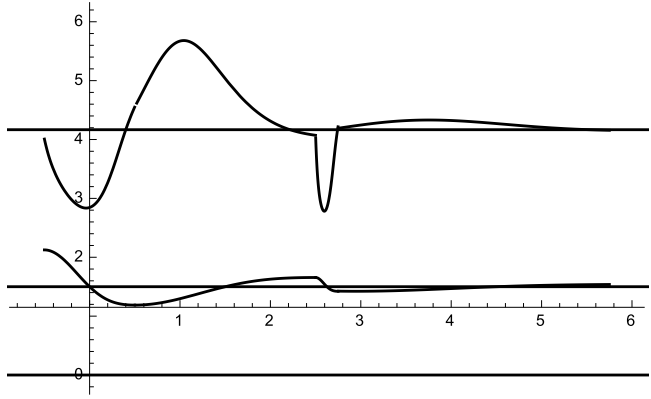} 
\caption{Example of lemma 2.4}\label{example2.6} 
\end{figure} 
If $4>a>2$, put $\alpha_k=\frac{k(k+1)\pi}{2}+\frac{a-2}{4-a}\frac{k}{k+1}\pi,\beta_k=\frac{k(k+1)\pi}{2}+\frac{a-2}{4-a}\frac{k+1}{k+2}\pi,k=0,1,\ldots$
\begin{equation}
x(t)=
\begin{cases}
2+\frac{1}{t}\cos{\frac{4-a}{a-2}(k+1)(k+2)(t-\alpha_k)} & t\in [\alpha_k,\beta_k)\\
2-\frac{1}{t}\cos{\frac{t-\beta_k}{k+1}} & t\in [\beta_k,\alpha_{k+1})
\end{cases}
\end{equation}

As in the example above, at $(\alpha_n+\beta_n)/2$ and $\beta_n$, this function will attain the two limits. This implies that $4$ is the best bound of $b$ when $4>a>2$.
\end{proof}

\begin{remark} If $a\geq 4$, then we cannot say more than the obvious inequality $b\geq a$, as seen by taking $\xi(t)\equiv a$ for which $x(t)\equiv(a+\sqrt{a^2-16})/2$ is a captured solution for $\xi$ making $b=a$.
\end{remark}

The proof of Theorem \ref{THM1} follows immediately from this lemma, since we know that for a Loewner process generated by a curve $\gamma$, $\gamma$ is simple if and only if $\lambda$ is uncaptured.

When $a=0$, we can also give a condition on the speed of convergence to $0$.
\begin{lemma}\label{brownian}
If $\lambda$ is captured at time $T$, $a=0$, $b>0$, and the limit satisfies
$$\varliminf_{t\rightarrow T^-}\frac{\lambda(T)-\lambda(t)}{\sqrt{T-t}}h(T-t)<C>0,$$
then we have
$$\varlimsup_{t\rightarrow T^-}\frac{\lambda(T)-\lambda(t)}{\sqrt{T-t}}\frac{1}{h(T-t)}> \frac{4}{C}$$
where $h$ is a positive function in $\mathbb{R}^+$  
satisfying 
$$\lim_{t\rightarrow 0^+}h(t)=\infty.$$
\end{lemma}
\begin{proof}
In the proof of the last lemma, use the same symbols, we can find a sequence $s_n\uparrow+\infty$ such that $\xi(s_n)<C/h(e^{-s_n})$. Hence there exists another sequence $t_n$ such that $s_n<t_n<s_n+1$ and $\xi(t_n)>4h(e^{-s_n})/C$. Put the time change back, we get the inequality immediately.
\end{proof}
As we mentioned before, when $\kappa>4$, the $\mathrm{SLE}_{\kappa}$-curve is not simple, meanning that $\lambda(t)=\sqrt{\kappa} B_t$ has captured times. We  define the time set 
$$I_{\kappa}\dot{=}\{t|\gamma_{\kappa}(t)\in \mathbb{R}\ \mathrm{or}\ \exists s<t, \mathrm{s.t.} \gamma_{\kappa}(s)=\gamma_{\kappa}(t)\}.$$
 It is easy to see that $I_{\kappa_1}\subset I_{\kappa_2}$ if $\kappa_1<\kappa_2$(a rigorous proof is similar as lemma \ref{COM}). Since the captured time must be a left local extreme point of the Brownian motion, the Lebesgue measure of $I_{\kappa}$ is $0$, which means that $I_{\kappa}$ is an exceptional set for Brownian motion. By L$\mathrm{\acute{e}}$vy's modulus of continuity of Brownian motion(see \cite{morters2010brownian}), almost surely, for all $T\in I_{\kappa}$, we have:
\begin{equation}\label{bulang}
    \varlimsup_{t\rightarrow T^-}\frac{|\sqrt{\kappa} B_T-\sqrt{\kappa} B_t|}{\sqrt{(T-t)\log(1/(T-t))}}\leq\sqrt{2\kappa}
\end{equation}
Using lemma \ref{brownian}, we deduce the following proposition about the local behaviour of the points in $I_{\kappa}$:
\begin{proposition}
If $\kappa>4$, then almost surely, for all $T\in I_{\kappa}$, we have:
$$\varliminf_{t\rightarrow T^-}\frac{|B(T)-B(t)|}{\sqrt{T-t}}\sqrt{\log\frac{1}{T-t}}\geq 2\frac{ \sqrt{2}}{\kappa}.$$
\end{proposition}
\begin{proof}
If this is not true, in lemma \ref{brownian}, we set $\lambda=\sqrt{\kappa}B$ and $h(T-t)=\log(1/(T-t))$. Then we get the inequality which is contradiction to \ref{bulang}.
\end{proof}
We give now an alternative condition on the driving function for $(P_1)$, based on the observation that $\xi$ can not be less than $4$ for a long time, because otherwise the solution $x$ would become negative.

We define the function $G$ as:
\begin{equation}\label{piece2}
    G(x)=
\begin{cases}
4-x & x\geq 2 \\
4/x & 0\leq x\leq 2
\end{cases}
\end{equation}

\begin{proposition}\label{test}
If $\exists \,t_2>t_1\geq0$ such that
\begin{equation}
\int_{t_1}^{t_2} G(\xi(s))\mathrm{d}s\geq\xi(t_1)
\end{equation}
then no solution $x$ with initial time $t_0\leq t_1$ is captured.
\end{proposition}
\begin{proof}
If $x$ is a captured solution, consider the function $\displaystyle h(t)=t-\frac{4}{c-t}$, $c$ is a positive constant, whose derivative is $\displaystyle h'(t)=1-\frac{4}{(c-t)^2}$, For $t\in (0,c)$,
$$h(t)\leq
\begin{cases}
h(c-2)=c-4 & c\geq 2 \\
h(0)=-\frac{4}{c} & 2>c\geq 0
\end{cases},$$
or, in other words, $h(t)\leq -G(c).$
Since $x(t)<\xi(t)$, we have
\begin{align*}
x(t_2)=\int_{t_1}^{t_2}\dot{x}(s)\mathrm{d}s+x(t_1)&\leq\int_{t_1}^{t_2}-G(x(s))\mathrm{d}s+x(t_1)\\
&\leq \int_{t_1}^{t_2}-G(\xi(s))\mathrm{d}s+\xi(t_1)\leq 0,
\end{align*}
which is impossible since $x(t)>0$, being captured. Hence there is no captured solution starting at a time $s<t_1$.
\end{proof}
By this proposition, if a Loewner process is generated by a curve and if the driving function satisfies the condition above, then the curve is simple. It is the case, for example, if $\xi(t)\leq 2,t\in(0,2)$. As we mentioned before, from  \cite{lawler2008conformally}, we know that if we already know that the Loewner equation is generated by a curve, then the curve is simple if and only if  the driving function is not captured after any time translation. Using this, we can prove half of Lind's H\"{o}lder-$\frac{1}{2}$ norm theorem directly:
\begin{corollary}
If the Loewner chain $(g_t)$ is generated by the curve $\gamma$, and the H\"{o}lder-$\frac{1}{2}$ norm of the driving function is less than 4, then $\gamma$ is a simple curve.
\end{corollary}
\begin{proof}
Assume that $\xi$ is the driving function after time change at time $T$. Since $\xi(t)<c<4$, for $t_1>0$ we can choose $t_2=t_1+c/(4-c)$, so that the condition of proposition \ref{test} holds. This means there is no captured solution at time $T$ since $t_1$ is arbitrary. Hence we get that $\xi$ is not captured at any time, implying that $\gamma$ is a simple curve.
\end{proof}
\subsection{An application of theorem \ref{THM1}:  Weierstrass functions}
In this section, we consider the Loewner equation with driving function $cW_b$ where $c$ is a positive constant and $W_b$ is the Weierstrass function 
\begin{equation}\label{WF}
    W_b(t)=\sum_{n=1}^{\infty}\frac{\cos(b^nt)}{\sqrt{b^n}}
\end{equation}

The local behavior of the Weierstrass function has been studied since long ago, see \cite{hardy1916weierstrass}. This function shares some properties with Brownian motion. In \cite{glenn2017loewner}, it is proven that 
\begin{equation}\label{WFN}
    \Vert W_b\Vert_{1/2}\leq\frac{b}{\sqrt{b}-1}+\frac{2}{1-\frac{1}{\sqrt{b}}}=C(b)\sim\sqrt{b},b\rightarrow+\infty
\end{equation}
By the main theorem of \cite{lind2005sharp}, if $c<4/C(b)\sim4/\sqrt{b}$, then the Loewner equation driven by $cW_b$ is generated by a quasislit curve, in particular a simple curve. Since theorem \ref{THM1} partly improves the result of \cite{lind2005sharp}, we may apply it to improve this result too.
\begin{theorem}\label{WFT}
$\forall l_0>1,\exists C>0$ s.t. if $c<C$, then the Loewner equation with driving function $cW_b$ is generated by a quasislit curve when $b>l_0$.
\end{theorem}
If $b$ is small, the proof of theorem \ref{WFT} follows easily from the methods of \cite{marshall2005loewner} and \cite{lind2005sharp}, so we may assume $b$ is large. Let $W_b^N$  be the $N^{\mathrm{th}}$ partial sum of the right side of (\ref{WF}). It is obvious that $\{W_b^N\}$ converge to $W_b$ uniformly on the real line when $N$ goes to infinity, and the estimate of the H\"{o}lder -$1/2$ norm above applies to $W_b^N$ as well.

We thus only need to prove that the Loewner equation which is driven by $cW^N_b$ is generated by a $K$  quasislit curve for all sufficiently large $N$. We first prove a lemma showing that $cW^N_b$ satisfies the hypothesis of theorem \ref{THM1} for large $N$.

\begin{lemma}\label{WFL}
For all $N$ and $T$ we have
$$\varliminf_{t\rightarrow T^-}\frac{|W^N_b(T)-W^N_b(t)|}{\sqrt{T-t}}<(\sqrt{\pi}+\frac{1}{\sqrt{\pi}})\frac{\sqrt{2}}{\sqrt{b}-1}\sim \frac{c_1}{\sqrt{b}}$$
\end{lemma}
\begin{proof}
Set $t_m=2\pi/b^{m-1}$,$m\in\mathbb{N}$, when $N>m-1$, we have
\begin{align*}
    &\frac{|W^N_b(T)-W^N_b(T-t_m)|}{\sqrt{t_m}}=\left|2\sum_{n=0}^{N}\frac{1}{\sqrt{b^n t_m}}\sin(b^nT-\frac{b^nt_m}{2})\sin(\frac{b^nt_m}{2})\right|\\
    =&2\left|\sum_{n=0}^{m-2}\frac{\sin(\frac{b^nt_m}{2})}{\sqrt{b^n t_m}}\sin(b^nT+\frac{b^nt_m}{2})+\sum_{n=m}^{N}\frac{\sin(\frac{b^nt_m}{2})}{\sqrt{b^n t_m}}\sin(b^nT+\frac{b^nt_m}{2})\right|\\
    <&\sum_{n=0}^{m-2}\frac{b^nt_m}{\sqrt{b^n t_m}}+2\sum_{n=m}^{N}\frac{1}{\sqrt{b^n t_N}}=\frac{\sqrt{2\pi}}{\sqrt{b}}\frac{1-\frac{1}{\sqrt{b^m}}}{1-\frac{1}{\sqrt{b}}}+\frac{\sqrt{2}}{\sqrt{\pi b}}\frac{1-\frac{1}{\sqrt{b^{N-m}}}}{1-\frac{1}{\sqrt{b}}}\\
    <&(\sqrt{\pi}+\frac{1}{\sqrt{\pi}})\frac{\sqrt{2}}{\sqrt{b}-1}
\end{align*}
when $N\leq m$, similar estimate shows that
$$\frac{|W^N_b(T)-W^N_b(T-t_m)|}{\sqrt{t_m}}<\frac{\sqrt{2\pi}}{\sqrt{b}-1}<(\sqrt{\pi}+\frac{1}{\sqrt{\pi}})\frac{\sqrt{2}}{\sqrt{b}-1}$$
which finishes the proof.
\end{proof}

Using the above lemma , we compute  $a(t)$ and $b(t)$ in theorem \ref{THM1}. For the driving function $cW_b^N$,
$$a(t)\leq c(\sqrt{\pi}+\frac{1}{\sqrt{\pi}})\frac{\sqrt{2}}{\sqrt{b}-1}\leq \frac{C_1c}{\sqrt{b}},b(t)\leq c\left(\frac{b}{\sqrt{b}-1}+\frac{2}{1-\frac{1}{\sqrt{b}}}\right)\leq C_2c\sqrt{b}$$
when $b>9$, for some constants $C_1,C_2$. Then $a(t)<2$ if $c<6/C_1$, and $b(t)<a(t)+\frac{4}{a(t)}$ if $b>\frac{2C_1}{C_2}$ and $c<\left(\frac{2}{C_1C_2}\right)^{1/3}$,
 which is the hypothesis of theorem \ref{THM1}. So if the Loewner equation in theorem \ref{WFL} is generated by a curve, this curve must be simple.

The rest of the proof needs the definition of conformal welding. If the Loewner equation is generated by a simple curve $\gamma$, then for all $t<T$, $g_T$ maps $\gamma(t)$ to two points in $\mathbb{R}$. Actually, $g_T(\gamma[0,T])=[X_1,X_2]$. We find a reverse homeomorphism $\phi:[X_1,\lambda(T)]\mapsto[\lambda(T),X_2]$ s.t. $g_T^{-1}(x)=g_T^{-1}(\phi(x))$. This function $\phi$ is called the conformal welding. In \cite{marshall2005loewner}, the authors proved that
if a Loewner equation is generated by a curve $\gamma$, then $\gamma$ is a quasislit curve if and only if $\exists M>0$ s.t. for all $T$, 
$$\frac{1}{M}<\frac{x-\lambda(T)}{\lambda(T)-\phi(x)}<M$$
for all $x$ and 
$$\frac{1}{M}<\frac{\phi(x)-\phi(y)}{\phi(x)-\phi(z)}<M$$
for all $\lambda(T)\leq x<y<z$ with $z-y=y-x$.

If the driving function is $W^N_b$, it is easy to prove that the H\"older-$1/2$ norm of $W^N_b$ in a sufficiently small interval is less than $4$. Hence the Loewner equation driven by $W^N_b$ is generated by a simple curve which is a juxtaposion of quasislit segments. Now, using the compactness argument as in \cite{marshall2005loewner},  we only need to proof that for all $N$, the conformal welding of the driving function $W^N_b$ satisfies the above two inequations with $M$  uniformly bounded in $N$.
We prove only the first one here since the proof of the second one is almost the same as in \cite{lind2005sharp}.
This proof needs two lemmas; the first one has independent interest and will be used in the proof of theorem 1.2.
\begin{lemma}\label{CAP}
Let $T>0$ be a fixed time, and $\lambda$  a driving function satisfying $\forall t\in[0,T), |\lambda(T)-\lambda(t)|/\sqrt{T-t}<C_1$. Then there exists a solution $x$ of the real Loewner equation (\ref{RLE}) defined in $[0,T]$, s.t. 
$$x(T)-\lambda(T)<(C_1+2)\sqrt{T}.$$
\end{lemma}
\begin{proof}
We consider the solution with initial value $x(0)=\lambda(T)+C\sqrt{T}$: since $x$ is increasing and $\lambda(t)<\lambda(T)+C\sqrt{T-t}<x(0)$, this solution exists in $[0,T]$. Let $\hat{x}$  be the solution of (\ref{RLE}) with driving function $\hat{\lambda}\equiv x(0)$ and initial value $\hat{x}(0)=x(0)^+$, then $\hat{x}(t)=x(0)+2\sqrt{t}$. Because $\hat{\lambda}>\lambda$, we have $x<\hat{x}$ which give us $x(T)<\hat{x}(T)=2\sqrt{T}+x(0)=(2+C)\sqrt{T}$.
\end{proof}
\begin{lemma}
If $c$ and $b$ satisfy the hypothesis of theorem \ref{WFT}, then for all $N$, there exists a constant $C_2>0$ s.t. for all $t<T$ and $x(0)>W^N_b(t)$, the solution $x$ of the real Loewner equation which is driven by $cW^N_b$ with initial value $x(t)=x_0$ satisfies 
$$x(T)-W_b^N(t)>C_2\sqrt{T-t}.$$
\end{lemma}
\begin{proof}
We claim that $x(t+(T-t)/b)>W^N_b(T)$ on the whole line. If $x(t)>W^N_b(T)$, then because $x$ is increasing and exists in all $\mathbb{R}$, the claim is true. If $x(t)\leq W^N_b(T)$ we can use the transformation and time change. It is easy to check that $W^N_b$ satisfy the condition of $\ref{THM1}$.The claim then follows from the proof of theorem \ref{THM1} and lemma  \ref{WFL}.

Since $(\Vert W^N_b\Vert_{1/2})_{N>0}$ are uniformly bounded by $C(b)$ , we have 
$$W^N_b(T)-W^N_b(s)<C(b)\sqrt{T-s},\forall s\in[t+(T-t)/b,T].$$
Let $l=x(t+(T-t)/b)-W^N_b(T)$: considering the solution $\hat{x}_l$ of (\ref{RLE}) driven by $-C(b)\sqrt{T-s}$ with initial value $\hat{x}_l(t+(T-t)/b)=l$, we have $x(T)-W_b^N(t)>C\sqrt{T-t}>\hat{x}_l(T)\geq\min_{l>0}\{\hat{x}_l(T)\}$.

So we only need to prove that the last term is larger than $C(T-t)$. There are two methods, the first one is by following the solution of remark \ref{solution}. The second one is by using the self-similar property of this special solution. We omit the details.
\end{proof}

Now let $C_1$ and $C_2$ to be the constants in the above two lemmas. Since $C_1$ and $C_2$  depend only on $b$, we may choose $M$ to be $(C_1+2)/C_2$, which finishes the proof of theorem \ref{WFT}.

\section{Imaginary Loewner Equation}
\subsection{Time change}
In this section we discuss the properties of the imaginary part of $g_t(z)$. More precisely, if we write $g_t(z)=X_t(z)+Y_t(z)i$, then the Loewner equation may be written as a couple of real ODEs:
\begin{equation}\label{TLE}
\begin{cases}
\displaystyle\dot{X}(t)=\frac{2(X(t)-\lambda(t))}{(X(t)-\lambda(t))^2+Y(t)^2}\\
\displaystyle\dot{Y}(t)=-\frac{2Y(t)}{(X(t)-\lambda(t))^2+Y(t)^2}
\end{cases}
\end{equation}
with initial value $X(0)=x_0,Y(0)=y_0>0$.
Setting $\theta(t)=X(t)-\lambda(t)$, the second equation of ~(\ref{TLE}) becomes
\begin{equation}\label{ILE}
\dot{Y}(t)=-\frac{2Y(t)}{\theta(t)^2+Y(t)^2}, Y(0)=y_0>0
\end{equation}
In the rest of this section, we consider a continuous function $\theta:\,[0,+\infty)\to \mathbb{R}_+$, and study the  corresponding equation (\ref{ILE}). We call this equation the imaginary Loewner equation with driving function $\theta$. As in the real case, we address the following question: is there a solution of (\ref{ILE}) that converges to $0$ in finite time?

\begin{definition}
If there exists an initial value $y_0>0$ and $T>0$ such that $Y(T)=0$ while $ Y(t)>0$ if $t<T$,where $Y$ is the solution of the equation with initial value $y_0$, we say that $\theta$ is a vanishing driving function at $T$ and that $Y$ a vanishing solution at $T$.
\end{definition}

This problem will be shown below to be connected to the second problem of the introduction. To study the equation (\ref{ILE}), we perform the same transformation as for the real equation. Namely, if we assume, as we may wlog that $T=1$, we set 
$$Y_-(t)=Y(t)/\sqrt{1-t},\theta_-(t)=\theta(t)/\sqrt{1-t}$$
The equation becomes 
$$\dot{Y}_-(t)=\frac{1}{2(1-t)}\left(Y_-(t)-\frac{4Y_-(t)}{\theta_-(t)^2+Y_-(t)^2}\right)$$

Using the same time change $\sigma(t)=1-e^{-2t}$, setting $y(t)=Y_-(\sigma(t))$ and $\eta(t)=\theta_-(\sigma(t))$, we have

\begin{equation}\label{CON1}
\dot{y}(t)=y(t)-\frac{4y(t)}{\eta(t)^2+y(t)^2},y(0)=y_0>0
\end{equation}

Let now $X$ be a solution of (\ref{RLE}); if (\ref{RLE}) has a captured solution $X_0$ at time 1, we define $W(t)=X(t)-X_0(t),\theta(t)=X_0(t)-\lambda(t)$. Then (\ref{RLE}) becomes
$$\dot{W}(t)=-\frac{2W(t)}{\theta(t)(W(t)+\theta(t))},W(0)=w_0$$
As before, we set
$$W_-(t)=W(t)/\sqrt{1-t},\theta_-(t)=\theta(t)/\sqrt{1-t}$$
and after the time change $\sigma(t)=1-e^{-2t},\,w(t)=W_-(\sigma(t)),\,\eta(t)=\theta_-(\sigma(t))$. The function $w$ obeys the following ODE:

\begin{equation}\label{CON2}
\dot{w}(t)=w(t)-\frac{4w(t)}{\eta^2(t)+\eta(t)w(t)}, w(0)=w_0
\end{equation}
The two equations (\ref{CON1}) and (\ref{CON2}) are very similar and share many properties which will help us in connection with problem $(P_2)$.

\subsection{Transition for the imaginary equation}
In this section, we consider the vanishing property of the imaginary equation. If $\eta\equiv 0$, it does vanish. So we only consider the case when the driving function is not identically $0$. Like in the real case, the vanishing property undergoes a phase-transition: 
\begin{theorem}\label{transition}
If $\eta(t)\geq2\sqrt{T-t}, \,t\in[0,T]$, then $\eta$ does not vanish at time $T$. If there exists a constant $C<2$ such that $\eta(t)<C\sqrt{T-t},\,t\in [0,T]$, then $\eta$ is  a vanishing driving function at time $T$.
\end{theorem}
The proof of this theorem requires two lemmas:
\begin{lemma}\label{COM}
Let $\eta_1$ and $\eta_2$ be two driving functions for the equation (\ref{CON1}). If $\forall t\geq0,\eta_1(t)\geq\eta_2(t)$ and if $\eta_1$ is vanishing at time $T$, then $\eta_2$ is also vanishing at time $T$. And the maximal vanishing solution of $\eta_1$ is not larger than the maximal vanishing solution of $\eta_2$.
\end{lemma}

\begin{proof}
We assume that $Y_1$ is a vanishing solution at time $T$ driven by $\eta_1$  with initial value $Y_1(0)=y_1$. If the solution with initial value $y_1$ and driving function $\theta_2$ is also vanishing at time $T$, then we are done. Otherwise, we have at least one solution driven by $\theta_2$ which is vanishing at a time $<T$. We consider the set $A$ of initial values which make the solutions driven by $\eta_2$ vanish at a time  $\leq T$. Let $a$  be the least upper-bound of $A$, $Y_2$ the solution with driving function $\eta_2$ and initial value $a$, and $T'$  be the vanishing time of $Y_2$, if $T'=T$ we are done.

Assume now $T'<T$. We consider the solution $Y'_2$ which is driven by $\eta_2$ with initial condition  $Y'_2(T')=Y_1(T')$. Since $Y'_2(T')=Y_1(T')>0=Y_2(T')$, it is easy to see that $Y'_2(0)>Y_2(0)=a$, hence $Y'_2(0)\notin A$. On the other hand, since $Y'_2(T')=Y_1(T')$, we have $Y'_2(t)\leq Y_1(t),\forall t\geq T'$, and $Y'_2(t)$ will vanish at $T$ or before $T$, so $Y'_2(0)\in A$, contrarily to the assumption.
\end{proof}

\begin{remark}
The analogue of lemma \ref{COM} for the real Loewner equation (\ref{HRLE}) is also true, and the proof is the same.
\end{remark}

\begin{lemma}\label{TRAN}
Let $c$ be a nonnegative number, $\varepsilon,T>0$ and let $y_{\varepsilon}$ be the solution of following equation:
\begin{equation}\label{imspecial}
\dot{y}=\frac{2y}{y^2+ct}, \ \ y(0)=\varepsilon.
\end{equation}
Then we have
$$\lim_{\varepsilon\rightarrow0}y_{\varepsilon}(T)=
\begin{cases}
\sqrt{T(4-c)} & c<4 \\
0 & c\geq4.
\end{cases}$$

\end{lemma}
\begin{proof}
We observe that (\ref{imspecial}) becomes linear if we consider $y$ as the variable and $t$ as the function:

$$\frac{\mathrm{d}t}{\mathrm{d}y}=\frac{c}{2y}t+\frac{y}{2}$$
If $c\neq4$, then we have $\displaystyle t=\frac{1}{4-c}y^2-\frac{1}{4-c}\varepsilon^{2-\frac{c}{2}}y^{\frac{c}{2}}$. Putting $t=T$, we have:
$\displaystyle \varepsilon^{\frac{c}{2}-2}=\frac{y(T)^{\frac{c}{2}}}{(c-4)T+y(T)^2}$, and it is easy to check that $\displaystyle \lim_{\varepsilon\rightarrow0}y_{\varepsilon}(T)=0$ for $c>4$ . When $c<4$, $\displaystyle \lim_{\varepsilon\rightarrow0}y_{\varepsilon}(T)=\sqrt{T(4-c)}$ because the left-hand side of the equality goes to $+\infty$.

If $c=4$, the solution is $\displaystyle t=\frac{1}{2}y^2\ln y-\frac{1}{2}y^2\ln\varepsilon$. Let $t=T$ as before: $\displaystyle\ln\varepsilon=\ln y(T)-\frac{2T}{y(T)^2}$. Since the left side tend to $-\infty$, then $\displaystyle \lim_{\varepsilon\rightarrow0}y_{\varepsilon}(T)=0$. This finishes the proof.
\end{proof}
Combining these two lemmas, we can now prove theorem  \ref{transition}:
\begin{proof}[Proof of theorem \ref{transition}]
By lemma \ref{COM}, we only need to prove that the driving function $t\mapsto 2\sqrt{T-t}$ is not vanishing at time $T$ while $t\mapsto c\sqrt{T-t}$ is if $c<2.$ From lemma \ref{TRAN}, we know that if $c=2$, then $\varepsilon\rightarrow y_{\varepsilon}(T)$ is a self homeomorphism of $(0,+\infty)$. Conversely, when the driving function is $t\mapsto 2\sqrt{T-t}$, any initial value will lead to a solution which is not $0$ at time $T$. When $c<2$, we let the initial value equals $\sqrt{T(4-c^2)}$, then the solution $Y_2(t)$ will be smaller than all the solutions $t\mapsto y_{\varepsilon}(T-t)$, hence we have $Y_2(T)<y_{\varepsilon}(0)=\varepsilon$ for arbitrary positive $\varepsilon$, implying that $Y_2$ is vanishing at $T$, and the proof is finished.
\end{proof}

\subsection{Properties of the Imaginary Equation}
In this subsection, we discuss (\ref{CON1}), the imaginary Loewner equation after time transformation. We write it as follows:
\begin{equation}\label{ILET}
    \frac{\dot{y}(t)}{y(t)}=1-\frac{4}{\eta(t)^2+y(t)^2}
\end{equation}
Using the time change, it is easy to check that $Y(t)>0$ is a vanishing solution if and only if $\displaystyle\lim_{t\rightarrow+\infty}e^{-t}y(t)=0$. 
\begin{definition}
If there exists a solution $y$ of (\ref{CON1}) s.t. $\displaystyle\lim_{t\rightarrow+\infty}e^{-t}y(t)=0$, then we say that the driving function $\eta$, the equation (\ref{CON1}) and the solution are vanishing. Otherwise, we say that the equation, the solution and the driving function are not vanishing.
\end{definition}
The condition on $y$ that implies the vanishing property may be considerably relaxed:
\begin{lemma}\label{4.4}
The solution $y$ is  vanishing  if and only if $0<y(t)<2,\forall t>0$.
\end{lemma}
\begin{proof}
If $y(t)<2$, then it is obvious that $\displaystyle\lim_{t\rightarrow+\infty}e^{-t}y(t)=0$. Conversely, assume that there exists $t_0$ s.t. $y(t_0)\geq 2$. From the equation (\ref{ILET}), $\dot{y}(t_0)>0$. Hence after a small time, $y(t)$ is larger than $2$. If $y(T)>2$, we have 
$$\frac{\mathrm{d}(y(s))}{\mathrm{d}s}>y(s)-\frac{4}{y(s)}, \forall s>T$$
Solving this differential inequation we get
$$y(s)>\sqrt{ce^{2t}+4}>\sqrt{c}e^t,$$
which means that $y$ is not a vanishing solution. 
\end{proof}

Even though the $e^{-t}$ condition is much weaker, we still use this condition as the definition in order to fit with (\ref{CON2}). 

Using the form of the equation (\ref{ILET}) we can provide an alternative proof of theorem \ref{transition}.
\begin{proof}
If $\eta(t)\geq 2,t\in[0,+\infty)$, the right side of (\ref{ILET}) is positive, which  means that $y$ is a increasing function. For all positive initial value $\varepsilon$, we have 
$$\frac{\mathrm{d}(\ln y(t))}{\mathrm{d}t}\geq \frac{\varepsilon}{4+\varepsilon}=C_{\varepsilon}$$
so that
$$y(\frac{1}{c_{\varepsilon}}\ln\frac{2}{\varepsilon})\geq y(0)\mathrm{exp}(\frac{1}{C_{\varepsilon}}\ln\frac{2}{\varepsilon}\cdot C_{\varepsilon})=2.$$
Hence there will be no vanishing solution, and this gives the proof of the first part.

If $\eta(t)<c<2, t\in[0,+\infty)$, we choose $y(0)=2-c$. Then $y$ will decrease on $[0,+\infty)$, so $y(t)<2$ holds. Applying lemma \ref{COM} and \ref{4.4} finishes the second part.
\end{proof}

These results maybe improved in several ways. We develop here one of them, that leads to the main idea of this paper. We define the "lower bound function" as
\begin{equation}\label{LB}
    L(t)=\int_0^t (1-\frac{4}{\eta(s)^2})\mathrm{d}s.
\end{equation}
This function plays a important role in the rest of this paper, because for any solution $y$ of equation (\ref{ILET}) and time interval $(a,b)$, we have 
$$\int_a^b \frac{\dot{y}(t)}{y(t)}\mathrm{d}t=\int_a^b( 1-\frac{4}{\eta(t)^2+y(t)^2})\mathrm{d}t>L(b)-L(a).$$
The next lemma gives a necessary condition on the lower bound function to make a driving function $\eta$  vanish.
\begin{lemma}
If $\eta$ is a vanishing driving function, then $\displaystyle\lim_{t\rightarrow+\infty}L(t)=-\infty$.
\end{lemma}
\begin{proof}
We first prove that $\displaystyle\lim_{t\rightarrow+\infty}L(t)$ exists. If $L(t)$ does not converge, then we can find two constants $C_1>C_2$ s.t. $L(t)$ takes these two value infinitely many times as $t\to+\infty$. Hence there exists an increasing sequence $\{a_n\}$ such that $L(a_n)=C_1$ and such that between time $a_n$ and $a_{n+1}$, there is a time $b_n$ s.t. $L(b_n)=C_2$. And without losing generality, we may assume that $b_n$ is the first time when $L(t)$ equal $C_2$ after $a_n$.

From the definition of $L(t)$, we see that
$$\int_{a_n}^{b_n}(1-\frac{4}{\eta(t)^2})\mathrm{d}t=L(b_n)-L(a_n)=C_2-C_1$$
Notice that $\displaystyle 1-\frac{4}{\eta(t)^2}>0$ when $\eta(t)>2$, so we can split this integral in two parts and drop the second one
\begin{align*}
C_2-C_1&=\left(\int_{(a_n,b_n)\cap\{\eta\leq 2\}}+\int_{(a_n,b_n)\cap\{\eta>2\}}\right)\left(1-\frac{4}{\eta(t)^2}\right))\mathrm{d}t \\
&\geq\int_{(a_n,b_n)\cap\{\eta\leq2\}}(1-\frac{4}{\eta(t)^2})\mathrm{d}t>-\int_{(a_n,b_n)\cap\{\eta\leq2\}}\frac{4}{\eta(t)^2}\mathrm{d}t
\end{align*}
Integrating equation (\ref{CON1}) from $a_n$ to $b_n$, we get 
\begin{align*}
\ln\left(\frac{y(b_n)}{y(a_n)}\right)&=\int_{a_n}^{b_n}\mathrm{d}(\ln(y(t)))=\int_{a_n}^{b_n}(1-\frac{4}{\eta(t)^2+y(t)^2})\mathrm{d}t \\
&=\int_{a_n}^{b_n}\left[1-\frac{4}{\eta(t)^2}+\frac{4}{\eta(t)^2}(1-\frac{1}{1+\frac{y(t)^2}{\eta(t)^2}})\right]\mathrm{d}t
\end{align*}
Assuming $y(0)=\varepsilon$, since  $L$ is the lower bound function, we see that for $t\in[a_n,b_n]$, we have
$$\ln\left(\frac{y(t)}{y(0)}\right)=\int_0^t\left(1-\frac{4}{\eta(t)^2+y(t)^2}\right) \mathrm{d}t\geq L(t)\geq L(b_n)=C_2,$$
which gives for $y(t)$ the lower bound $\varepsilon e^{C_2}=C_3$. We can then  estimate the variation of $y(t)$ in $[a_n,b_n]$ as follows:
\begin{align*}
\ln\left(\frac{y(b_n)}{y(a_n)}\right)&>C_2-C_1+\int_{(a_n,b_n)\cap\{\eta\leq2\}}\frac{4}{\eta(t)^2}\left(1-\frac{1}{1+\frac{y(t)^2}{\eta(t)^2}}\right)\mathrm{d}t \\
&\geq C_2-C_1+\int_{(b_n,a_n)\cap\{\eta\leq2\}}\frac{4}{\eta(t)^2}\left(1-\frac{1}{1+\frac{C_3^2}{2^2}}\right)\mathrm{d}t\\
&\geq C_2-C_1+C_4(C_1-C_2)
\end{align*}
where $0<\displaystyle C_4=1-\frac{1}{1+\frac{C^2_3}{4}}<1$ depends only on $C_1,C_2$ and $\varepsilon$. But
$$\ln\left(\frac{y(a_{n+1})}{y(a_n)}\right)=\ln\left(\frac{y(a_{n+1})}{y(b_n)}\right)+\ln\left(\frac{y(b_n)}{y(a_n)}\right)$$
the first term being greater than $L(a_{n+1})-L(b_n)=C_1-C_2$, we get that
$$y(a_{n+1})>y(a_n)e^{C_4(C_1-C_2)}=C_5 y(a_n).$$
Since $C_5>1$, for arbitrary $\varepsilon$, when $n$ is sufficiently large, $y(a_n)>2$. That means that $\eta(t)$ is not a vanishing driving function. This contradiction leads to the conclusion that the limit of $L(t)$ must exist.

We now prove that $\displaystyle\lim_{t\rightarrow+\infty}L(t)=-\infty$. Otherwise this limit is a constant $C_6$, and $\exists C_7$ s.t. $L(t)>C_7, \forall t>0$. Just like before, assuming $y(0)=\varepsilon$, we see that $y(t)>y(0)e^{L(t)}>\varepsilon e^{C_7}=C_8.$ Recalling that
$$\ln\left(\frac{y(t)}{y(0)}\right)=L(t)+\int_0^t\frac{4}{\eta(s)^2}\left(1-\frac{1}{1+\frac{y(s)^2}{\eta(s)^2}}\right)\mathrm{d}s,$$
we may write 
\begin{align*}
\ln\left(\frac{y(t)}{y(0)}\right)&\geq L(t)+\int_{(0,t)\cap\{\eta<3\}}\frac{4}{\eta(s)^2}\left(1-\frac{1}{1+\frac{y(s)^2}{\eta(s)^2}}\right)\mathrm{d}s\\
&\geq L(t)+\int_{(0,t)\cap\{\eta<3\}}\frac{4}{3^2}\left(1-\frac{1}{1+\frac{C_8^2}{3^2}}\right)\mathrm{d}s\\
&=L(t)+C_9\vert (0,t)\cap\{\eta<3\}\vert
\end{align*}
where $\vert A\vert$ denotes the Lebesgue measure of the Borel set $A$. If $\vert(0,+\infty)\cap\{\eta<3\}\vert=+\infty$, then $y(t)>2$ for $t$ large, which leads to a contradiction. Hence we may assume that $\vert(0,+\infty)\cap\{\eta<3\}\vert=C_{10}<+\infty$ and consequently that $\vert(0,+\infty)\cap\{\eta\geq3\}\vert=\infty$. But from the definition of $L(t)$
$$L(t)=t-\left(\int_{(0,t)\cap\{\eta<3\}}+\int_{(0,t)\cap\{\eta\geq3\}}\right)\frac{4}{\eta(s)^2}\mathrm{d}s$$
$$\int_{(0,t)\cap\{\eta<3\}}\frac{4}{\eta(s)^2}\mathrm{d}s>t-L(t)-\frac{4}{9}l((0,+t)\cap\{\eta\geq3\})>\frac{5}{9}t-L(t).$$
Going back to the inequality before, we have
$$\int_{(0,t)\cap\{\eta<3\}}\frac{4}{\eta(s)^2}\left(1-\frac{1}{1+\frac{y(s)^2}{\eta(s)^2}}\right)\mathrm{d}s>\left(1-\frac{1}{1+\frac{C^2_8}{9}}\right)\int_{(0,t)\cap\{\eta<3\}}\frac{4}{\eta(s)^2}\mathrm{d}s,$$
from which it follows that $y(t)\to +\infty$ since
$$\ln\left(\frac{y(t)}{y(0)}\right)\geq L(t)+C_{11}\frac{5}{9}t-L(t)=C_{12}t.$$
It follows that $y$ is not a vanishing solution, contradicting the assumption: the lemma is proven.
\end{proof}

Although the lower bound function of a vanishing driving function tends to $-\infty$, there still exist vanishing solutions that do not tends to $0$, a simple example being $\eta(t)=y(t)=\sqrt{2}$. Hence we would like to investigate the behaviour of vanishing solutions when time goes to $+\infty$.
\begin{lemma}\label{LIM2}
There is at most one vanishing solution satisfying $\displaystyle\lim_{t\rightarrow+\infty}y(t)\neq0$.
\end{lemma}
\begin{proof}
Assume that $y_1>y_2$ are two vanishing solutions. Substracting the two equations (\ref{CON1}) we have
\begin{equation*}
\frac{\mathrm{d}(y_1(t)-y_2(t))}{\mathrm{d}t}=(y_1(t)-y_2(t))-\frac{4(y_1(t)-y_2(t))(\eta(t)^2-y_1(t)y_2(t))}{(\eta(t)^2+y_1(t)^2)(\eta(t)^2+y_2(t)^2)},
\end{equation*}
putting $u(t)=y_1(t)-y_2(t)$, we get 
\begin{equation}\label{LIM}
\frac{\dot{u}(t)}{u(t)}=1-\frac{1}{1+\frac{y_1(t)^2}{\eta(t)^2}}\frac{4}{\eta(t)^2+y_2(t)^2}+\frac{4y_1(t)y_2(t)}{(\eta(t)^2+y_1(t)^2)(\eta(t)^2+y_2(t)^2)}.
\end{equation}
 There are three cases: 1) $\displaystyle\lim_{t\rightarrow+\infty}y_2(t)$ does not exist, 2)  $\displaystyle\lim_{t\rightarrow+\infty}y_2(t)=c>0$, or 3)  $\displaystyle\lim_{t\rightarrow+\infty}y_2(t)= 0$. 

In the first case, we find two constant $C_1>C_2$, and two sequences $(a_n)$ and $(b_n)$ converging to $+\infty$ s.t. $y_2(a_n)=C_1, y_2(b_n)=C_2,a_n<b_n<a_{n+1}<b_{n+1},$ and $b_n$ is the first time that $y_2$ is equal to $C_2$ after $a_n$, $\forall n>0$. Integrating the last term of (\ref{LIM}) from $a_n$ to $b_n$, we have 
\begin{align*}
    \int_{a_n}^{b_n}\frac{4y_1(t)y_2(t)\mathrm{d}t}{(\eta^2+y_1(t)^2)(\eta^2+y_2(t)^2)}&\geq\int_{(a_n,b_n)\cap\{\eta<2\}}\frac{4C_2^2\mathrm{d}t}{(2^2+2^2)(2^2+C_2^2)} \\
    &=C_3|(a_n,b_n)\cap\{\eta<2\}|
\end{align*}
From equation (\ref{CON1}) for $y_2$
\begin{align*}
\ln\frac{C_2}{C_1}&=\int_{a_n}^{b_n}\left(1-\frac{4}{\eta(t)^2+y(t)^2}\right)\mathrm{d}t>\int_{(a_n,b_n)\cap\{\eta<2\}}\left(1-\frac{4}{\eta(t)^2+y(t)^2}\right)\mathrm{d}t\\
&>\int_{(a_n,b_n)\cap\{\eta<2\}}\left(1-\frac{4}{C_2^2}\right)\mathrm{d}t=|(a_n,b_n)\cap\{\eta<2\}|\left(1-\frac{4}{C^2_2}\right)
\end{align*}
Since $1-4/C_2^2$ is negative, we have 
$$|(a_n,b_n)\cap\{\eta<2\}|>\frac{C^2_2}{C^2_2-4}\ln \frac{C_2}{C_1}=C_4>0.$$
Going back to (\ref{LIM}), integrating from $a_1$ to $b_n$,
\begin{align*}
\ln\frac{u(b_n)}{u(a_1)}&>\int_{a_1}^{b_n}1-\frac{4}{\eta(t)^2+y_2(t)^2}\mathrm{d}t+\sum_{i=1}^n\int_{a_i}^{b_i}\frac{4y_1(t)y_2(t)\mathrm{d}t}{(2^2+y_1(t)^2)(2^2+y_2(t)^2)}\\
&>\ln\frac{y_2(b_n)}{y_2(a_1)}+nC_3\frac{C^2_2}{C^2_2-4}\ln\frac{C_2}{C_1}=\ln\frac{C_2}{C_1}+nC_3C_4
\end{align*}
It follows that $\lim_{n\rightarrow+\infty}y_1(b_n)=+\infty$, and we have a contradiction since $y_1$ is a vanishing solution. 

If now $\displaystyle \lim_{t\to+\infty}y_2(t)=c>0$, we use the same method, and first prove  that $|(0,+\infty)\cap\{\eta<2\}|=+\infty$ as above. Integrating (\ref{CON1}) with $y_2$,
\begin{align*}
    \ln\frac{y_2(t)}{y_2(0)}&=t-\int_0^t\frac{4\mathrm{d}s}{\eta(s)^2+y_2(s)^2}\\
    &>t-\int_{(0,t)\cap\{\eta\geq 2\}}\frac{4\mathrm{d}s}{4+y_2(s)^2}-\int_{(0,t)\cap\{\eta<2\}}\frac{4\mathrm{d}s}{y_2(s)^2}
\end{align*}
Let $t$ tend to infinity: if $|(0,+\infty)\cap\{\eta<2\}|<+\infty$, then the left side of the inequality is bounded while the right one goes to infinity as $\frac{c^2}{4+c^2}t$, thus leading to a contradiction. Now take $A$ large enough so that $y(t)\geq c/2,\,t\geq A$. Then, as above,
$$\int_{A}^{t}\frac{4y_1(s)y_2(s)\mathrm{d}s}{(\eta^2(s)+y_1(s)^2)(\eta^2(s)+y_2(s)^2)}\geq \frac{c^2}{64}\vert [A,t]\cap\{\eta<2\}\vert,$$
and 
$$\ln{(\frac{u(t)}{u(A)})}\geq \ln{(\frac{y_2(t)}{y_2(A)})}+\frac{c^2}{64}\vert [A,t]\cap\{\eta<2\}\vert,$$
which is impossible. The only left possibility is that $y_2(t)\to 0$ as $t\to +\infty$, and the claim is proven.
\end{proof}

There are some driving functions whose captured solutions all tend to $0$. An example is when $\eta$ is less than $2$ but converge to $2$. Actually, this lemma shows that all the vanishing solutions of (\ref{CON1}) converge to $0$ with at most one possible exception, which is then the greatest vanishing solution.

\subsection{Basic property of the dual equation (\ref{CON2})}
As we mentioned before, the vanishing property of this equation is similar to (\ref{CON1}).
\begin{definition}
We say the driving function $\eta(t)$ and equation (\ref{CON2}) are vanishing if there exists a solution $w(t)$ s.t. $\lim_{t\rightarrow+\infty}e^{-t}w(t)=0$, and $w$ vanishing solution.
\end{definition}
The transition for equation (\ref{CON2}) is the same as for (\ref{CON1}). We omit the proof since all the details are as same as the second proof of theorem \ref{transition}.
\begin{lemma}
If $\eta(t)\geq 2,\forall t\geq0$, then the equation (\ref{CON2}) is not vanishing. If $\eta(t)<C<2,\forall t\geq 0$, then the equation (\ref{CON2}) is vanishing.
\end{lemma}

The second property is similar to lemma \ref{LIM2}, which states that most of the solutions converge to $0$. But the statement is different since (\ref{CON2}) may have unbounded vanishing solution, and the proof also needs to be modified.

\begin{lemma}\label{LIM3}
Equation (\ref{CON2}) has at most one vanishing solution such that $\displaystyle\varlimsup_{t\rightarrow+\infty}w(t)>0$.
\end{lemma}
\begin{proof}
The  proof is similar to that of lemma \ref{LIM2}. Assume $w_1>w_2$ are two solutions of (\ref{CON2}), and that $w_1$ satisfies $\displaystyle\varlimsup_{t\rightarrow+\infty}w_1(t)=C>0$: we want to prove that $\displaystyle\lim_{t\rightarrow+\infty}w_2(t)=0$.\\
Subtracting  these solutions as before, we get
\begin{align*}
\dot{w_1}(t)-\dot{w_2}(t)=w_1(t)-w_2(t)-\frac{4(w_1(t)-w_2(t))}{(\eta(t)+w_1(t))(\eta(t)+w_2(t))}.
\end{align*}
Set $v(t)=w_1(t)-w_2(t)$: we have 
\begin{align*}
\frac{\dot{v}(t)}{v(t)}&=1-\frac{4}{(\eta(t)+w_1(t))(\eta(t)+w_2(t))}\\
&=1-\frac{4}{\eta(t)(\eta(t)+w_1(t))}+\frac{4w_2(t)}{\eta(t)(\eta(t)+w_1(t))(\eta(t)+w_2(t))}
\end{align*}
Integrating this equality from $0$ to $t$ we get
\begin{align*}
    \ln\frac{v(t)}{v(0)}&=\ln\frac{w_1(t)}{w_1(0)}+\int_0^{t}\frac{4\mathrm{d}s}{\eta(s)(\eta(s)+w_1(s))(1+\frac{\eta(s)}{w_2(s)})}\\
    &=\ln\frac{w_1(t)}{w_1(0)}+\int_0^{t}M(s)\mathrm{d}s
\end{align*}
We only need to prove that the last term is unbounded as $t$ goes to $+\infty$. If not, we have $w_2(t)/w_2(0)=c(t)w_1(t)/w_1(0)$, where $c$ is a decreasing function converging to a constant $C_1<1$. Just like in lemma \ref{LIM2}, there are two cases: the limit of $w_2$ as $t\rightarrow+\infty$ does not exist or it is a positive number.

In the first case, we choose two sequences $\{a_n\}$ and $\{b_n\}$ increasing to infinity as before: $w_2(a_n)=C_2, w_2(b_n)=C_3, C_2>C_3>0, a_n<b_n<a_{n+1}$ and $b_n$ is the first time after $a_n$ such that $w(b_n)=C_3$. For sufficiently large $n$, we have $w_1(a_n)/w_1(b_n)>\sqrt{C_2/C_3}=C_4$, so that
\begin{align*}
    \int_{a_n}^{b_n}M(t)\mathrm{d}t&>\int_{(a_n,b_n)\cap\{\eta<2\}}\frac{4\mathrm{d}t}{\eta(t)(\eta(t)+w_1(t))(1+\frac{2}{C_3})}\\
    &>C_5\int_{(a_n,b_n)\cap\{\eta<2\}}\left(\frac{4}{\eta(t)(\eta(t)+w_1(t))}-1\right)\mathrm{d}t\\
    &>C_5\int_{a_n}^{b_n}\left(\frac{4}{\eta(t)(\eta(t)+w_1(t))}-1\right)\mathrm{d}t\\
    &=-C_5\ln\frac{w_1(b_n)}{w_1(a_n)}>-C_5\ln\frac{1}{C_4}=C_6>0
\end{align*}
implying that the integral of $M$ is unbounded.

In the second case, estimating $|(0,+\infty)\cap\{\eta<2\}|$ as in lemma {\ref{LIM2}}, we arrive at the same conclusion.
\end{proof}
\begin{remark}
Equation (\ref{CON2}) may have several vanishing solution which are unbounded. And it is easy to prove that if (\ref{CON2}) has an unbounded solution then the driving function must satisfy $\displaystyle\varliminf_{t\rightarrow+\infty}\eta(t)=0$. This condition will play an important role later. Actually, this property tells us that, except maybe for the smallest one, if $x_1$ and $x_2$ are two captured solution which are driven by $\xi$, then $\displaystyle\lim_{t\rightarrow+\infty}|x_1(t)-x_2(t)|=0$.
\end{remark}
The next question is: what kind of driving function will make (\ref{CON2}) vanishing but (\ref{CON1}) not?
\begin{proposition}\label{LIM4}
If a bounded driving function $\eta$ is vanishing in equation (\ref{CON2}) but is not in equation (\ref{CON1}), then $\displaystyle\varliminf_{t\rightarrow+\infty}\eta(t)=0$.
\end{proposition}
\begin{proof}
We assume that $\eta$ has a positive lower bound $c>0$. From lemma \ref{LIM3}, we can choose a vanishing solution of (\ref{CON2}) $w(t)$ s.t. $w(t)<c\leq\eta(t)$. Then we consider the solution $y$ of (\ref{CON1}) with initial value $y(0)=w(0)$. We have 
\begin{align*}
    \frac{\mathrm{d}(\ln y)}{\mathrm{d}t}=1-\frac{4}{\eta(t)^2+y(t)^2}<1-\frac{4}{\eta(t)^2+cy(t)}\leq1-\frac{4}{\eta(t)^2+\eta(t)y(t)}
\end{align*}
hence $y$ is smaller than $w$ and $y$ is a vanishing solution of (\ref{CON1}), this finishes the proof. 
\end{proof}

\section{Proof of theorem \ref{THM3}}

As already mentioned, solving ($P_2$) is linked to the study of the case  $$\displaystyle\lim_{z\rightarrow\lambda(T)}g^{-1}_T(z)$$ does not exist. If it happens, we have four cases:

\begin{enumerate}[(a)]
\item There are at least two limit points which are accessible.
\item There is no accessible limit point.
\item There is only one accessible point $x$ and $x\in \mathbb{R}$.
\item There is only one accessible point $z$ and $z\in \mathbb{H}$.
\end{enumerate}

\begin{figure}[h]\label{example}
\centering
\subfigure[two accessible points]{
\label{Fig.sub.1}
\includegraphics[width=0.4\textwidth]{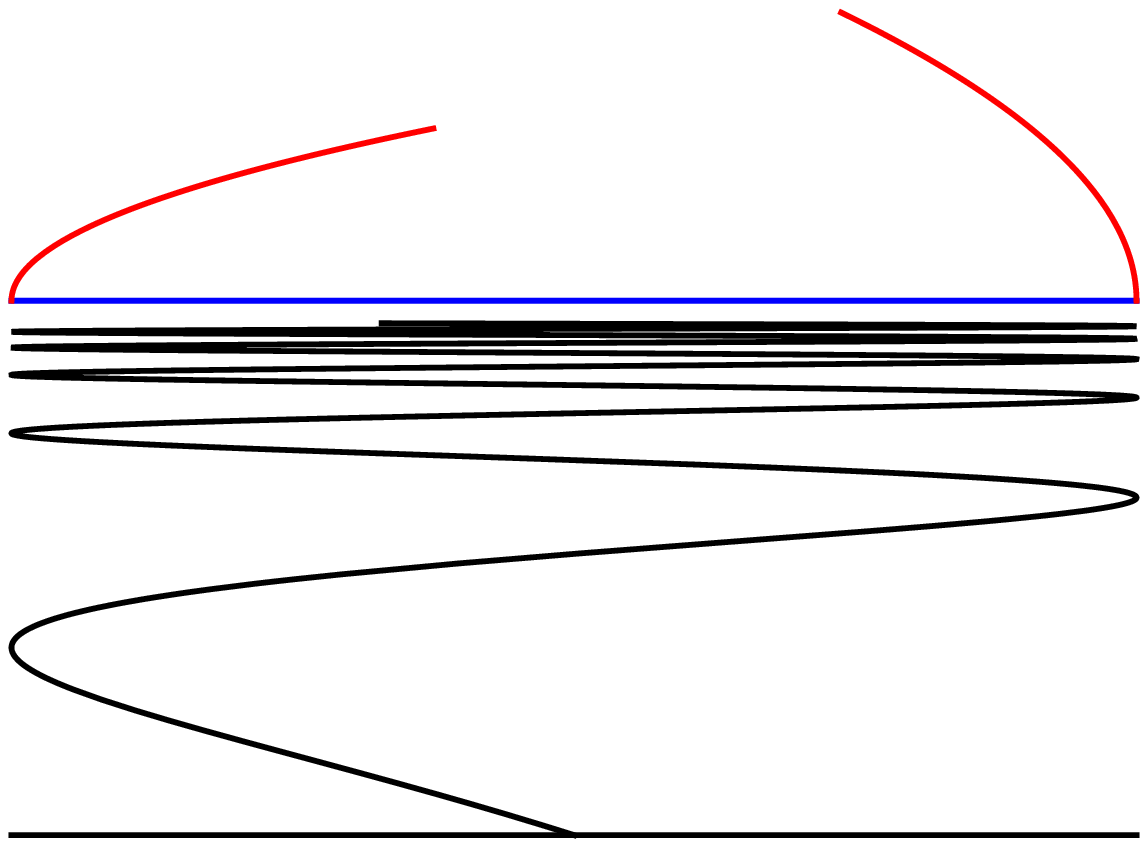}}
\subfigure[no accessible point]{
\label{Fig.sub.2}
\includegraphics[width=0.4\textwidth]{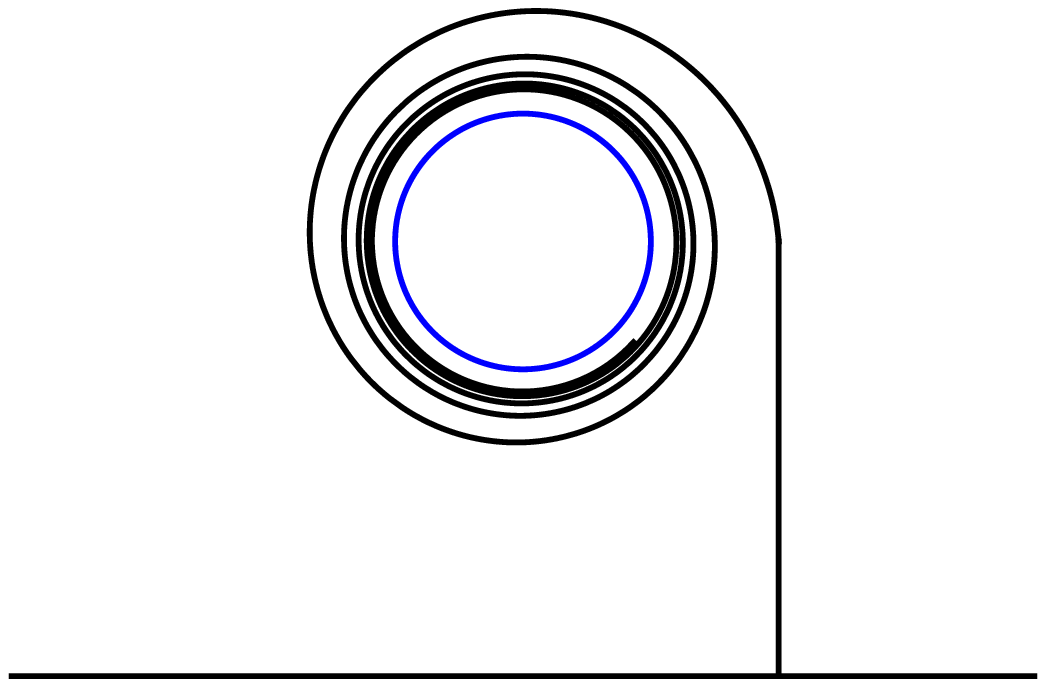}}
\subfigure[one accessible point in $\mathbb{R}$]{
\label{Fig.sub.3}
\includegraphics[width=0.4\textwidth]{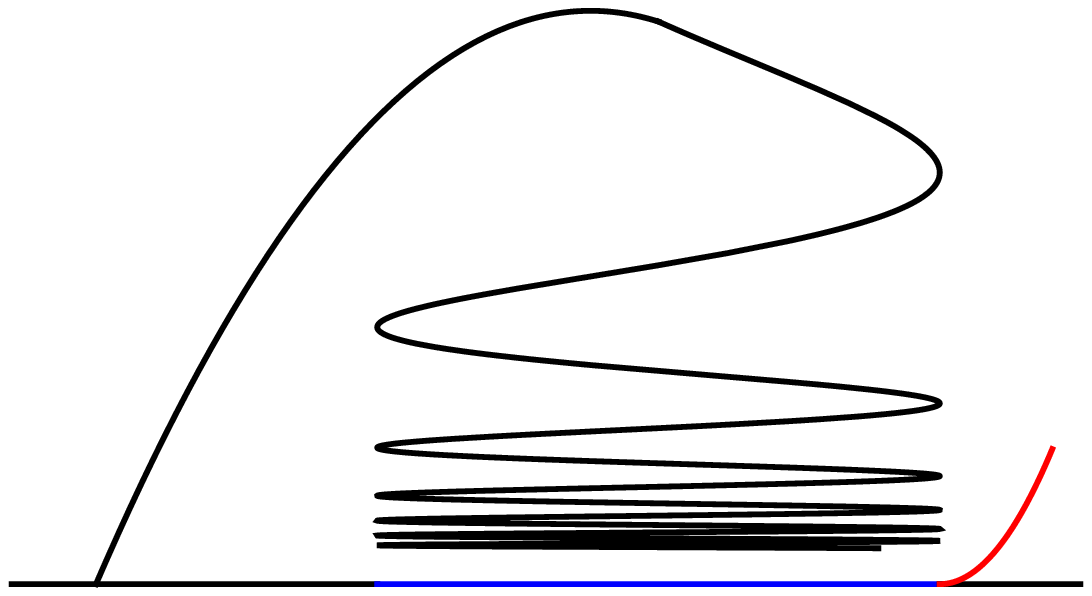}}
\subfigure[one accessible point in $\mathbb{H}$]{
\label{Fig.sub.4}
\includegraphics[width=0.4\textwidth]{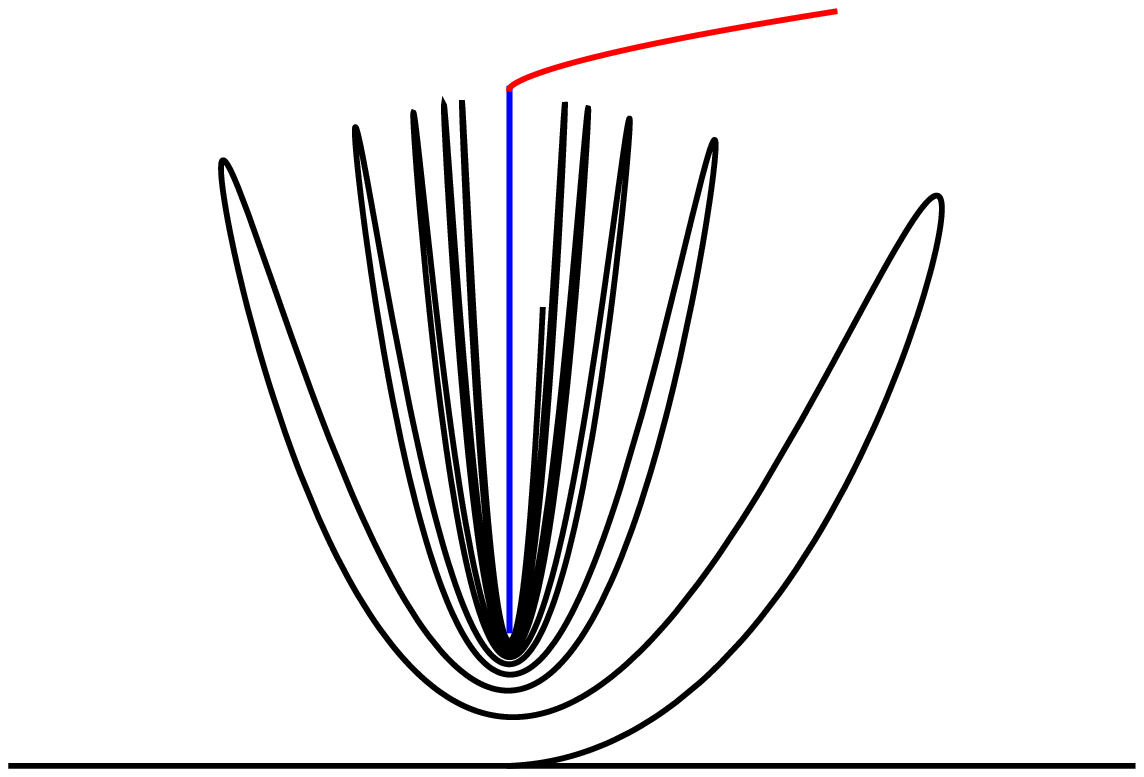}}
\caption{examples of the Loewner equation is not generated by curve}
\label{Fig.lable}
\end{figure}

Figure \ref{example} illustrates these $4$ cases. In this figure, the blue lines are the set of limit points and the red line is the path $\beta$ of the definition of accessible point.

In \cite{pommerenke1966loewner}, it has been proved that if (a) holds, then the driving function is not continuous at time $T$. The analogous property of driving function if (b) holds is unknown except for a special example given in \cite{lind2010collisions}.

Our purpose is to discuss (c) and (d). About (c), like before, we may, w.l.o.g, assume $T=1$ . In figure (c) of \ref{example}, the limit points of $g^{-1}_1(z)$ form an interval of the real axis. Every limit point $x$ is the initial point of   captured solution at time $1$ of(\ref{RLE}). And every captured solution will give a driving function $\eta$ of (\ref{CON2}) after time change. From (\ref{RKEY}), we have that 
$$e^{-t}(x(t)-\dot{x}(t))=4e^{-t}/\eta(t).$$
 Integra this identity from $t$ to $+\infty$, we find the correspondence between $\xi$ and $\eta$ :
\begin{equation}\label{etaxi}
    \xi(t)=\eta(t)+\int_{0}^{+\infty}\frac{4e^{-s}}{\eta(t+s)}\mathrm{d}s
\end{equation}

Let $C_+[0,+\infty)$ be the space of positive continuous function in $[0,+\infty)$, and 
$$\mathcal{V}\dot{=}\{\eta\in C_+[0,+\infty)|\int_0^\infty e^{-s}/\eta(s)ds<\infty\}.$$
Like $\mathcal{I}$ in theorem \ref{RSTRU}, $\mathcal{V}$ is a cone.

We define an operator $H$ on $\mathcal{V}$ by:
$$H(\eta)(t)=\eta(t)+\int_{0}^{+\infty}\frac{4e^{-s}}{\eta(t+s)}\mathrm{d}s.$$
If $\xi$ is a captured driving function at $T$ after time change, then there might be several $\eta$s satisfying $H(\eta)=\xi$. From the transformation and time change, every such $\eta$ corresponds to a captured solution at time $T$, and a captured solution corresponds to a initial value $x$ in $\mathbb{R}$ s.t. $T_x=T$, the mapping between $\eta$ and $x$ is $\eta\mapsto\lambda(0)-\eta(0)$ when $T=1$. We write $A=\{\eta|H(\eta)=\xi\}$ and $X=\{\lambda(0)-\eta(0)|\eta\in A\}$, then $X$ is a single point or a right closed interval($(x_1,x_2]$ or $[x_1,x_2]$).

\begin{lemma}\label{several}
Let $\xi$ be a captured driving function of (\ref{HRLE}). If there exists $\eta_0\in A$ satisfying $\displaystyle\varliminf_{t\rightarrow+\infty}\eta_0(t)>0$, and $\eta_0$ is a vanishing driving function of equation (\ref{CON2}), then $\displaystyle\varliminf_{t\rightarrow+\infty}\eta(t)>0,\forall\eta\in A$.
\end{lemma}
\begin{proof}
If $A$ consists of only one function, then the lemma is true. So we assume there are at least two functions in $A$. As we mentioned before, every function in $A$ corresponds to an interval in $X=[x_1,x_2]$ or $(x_1,x_2]$, and let $x_0\in X$ be the limit point corresponding to $\eta_0$. Because $\eta_0$ is a vanishing driving function of (\ref{CON2}), we have $x_0<x_2$. We choose two points $\hat{x}_1,\hat{x}_2$ from $X$ and $\hat{x}_1<\hat{x}_2<x_2$, $\hat{\eta}_1$ and $\hat{\eta}_2$ are the corresponding functions in $A$. Let $x$ be the solution of (\ref{CON2}) with initial value $\hat{x}_2-\hat{x}_1$, and $\hat{\eta}_1$ the corresponding driving function. Then $x$ is a captured solution, and $\hat{\eta}_1(t)=\hat{\eta_2(t)}+x(t)$. By lemma \ref{LIM3}, we have $\displaystyle\lim_{t\rightarrow+\infty}x(t)=0$. Notice that $\hat{x}_1$ and $\hat{x}_2$ can be all the points in $X\backslash\{x_2\}$ which includes $x_0$, which finishes the proof.
\end{proof}
\begin{lemma}\label{LPIR}
Let $I=\{x\in\mathbb{R}|T_x=1\}$. If $I$ is an interval $[x_1,x_2]$, and there exists $x_0\in [x_1,x_2)$ and $c>0$ s.t. the corresponding function $\eta_0$ satisfies $\eta_0(t)>c,\forall t$, then for all $x\in(x_1,x_2)$, $\exists\varepsilon>0$ s.t. $B(x,\varepsilon)\cap\Bar{\mathbb{H}}\subset J_1\dot{=}K_1\backslash\cup_{t<1}K_t$, where $B(x,\varepsilon)$ is the disk of center $x$ and radius $\varepsilon$.
\end{lemma}
\begin{proof}
We consider the solutions $X_{\varepsilon}(t),Y_{\varepsilon}(t)$ of (\ref{TLE}) with initial value $(X_{\varepsilon}(0)=x,Y_{\varepsilon}(0)=\varepsilon)$. Performing the same time change as before we have 
\begin{equation*}
    \dot{x}_{\varepsilon}=x_{\varepsilon}-\frac{4(\xi-x_{\varepsilon})}{(\xi-x_{\varepsilon})^2+y_{\varepsilon}^2}
\end{equation*}
where $x_{\varepsilon},y_{\varepsilon},\xi$ correspond to $X_{\varepsilon},Y_{\varepsilon},\lambda$ after time change respectively.

We assume $\hat{x}$ is the captured solution of (\ref{HRLE}) with initial value $\lambda(1)-x$. Define $\eta=\xi-\hat{x}$: by last lemma, we may assume $\eta(t)>C>0$. Set $w_{\varepsilon}=x_{\varepsilon}-\hat{x}$, then the equation becomes
\begin{align}
        \dot{w}_{\varepsilon}&=w_{\varepsilon}-\frac{4w_{\varepsilon}}{(\eta-w_{\varepsilon})\eta}+\frac{4}{(\eta-w_{\varepsilon})^2+y^2_{\varepsilon}}\frac{y^2_{\varepsilon}}{\eta-w_{\varepsilon}}\\
        &=w_{\varepsilon}-\frac{4w_{\varepsilon}}{(\eta+w_{\varepsilon})\eta}-\frac{4}{\eta-w_{\varepsilon}}\left(\frac{2w^2_{\varepsilon}}{\eta^2+\eta w_{\varepsilon}}-\frac{y^2_{\varepsilon}}{y^2_{\varepsilon}+(\eta-w_{\varepsilon})^2}\right)
\end{align}
with the initial value  $w_{\varepsilon}(0)=0$. Let $w^{\varepsilon}$  be the solution of (\ref{CON2}) with initial condition $w^{\varepsilon}(0)=\varepsilon$: we claim that for sufficient small $\varepsilon$, $w_{\varepsilon}(t)<w^{\varepsilon}(t)<\eta(t), \forall t\in[0,+\infty)$.

Let $y^{\varepsilon}$ be the solution of (\ref{CON1}) with initial value $y^{\varepsilon}(0)=\varepsilon$: by lemma \ref{LIM2} and proposition \ref{LIM4}, for sufficient small $\varepsilon$, $y^{\varepsilon}(t)<w^{\varepsilon}(t)<C/5,\forall t$. Notice that $y^{\varepsilon}$ is driven by $\eta=\xi-\hat{x}$ and $y_{\varepsilon}$ by $|\xi-x_{\varepsilon}|$. Since $x_{\varepsilon}(t)>\hat{x}(t),\forall t>0$, we see that $y_{\varepsilon}<y^{\varepsilon}$ holds before the first time $t$ s.t. $\xi(t)=\hat{x}(t)$, which is equivalent to saying that $w_{\varepsilon}(t)=\eta(t)$.

At time $0$, $w_{\varepsilon}(t)<w^{\varepsilon}(t)<C/5$. Define $t_1$ to be the first time when $w_{\varepsilon}(t)=w^{\varepsilon}(t)$. If $t_1<+\infty$, then $w_\varepsilon(t_1)=w^\varepsilon(t_1)>y^\varepsilon(t_1)>y_\varepsilon(t_1)$, and $y^2_{\varepsilon}(t_1)+(\eta(t_1)-w_{\varepsilon}(t_1))^2>16\eta^2(t_1)/25>3\eta^2(t_1)/5>\eta(t_1)^2/2+\eta(t_1) w_{\varepsilon}(t_1)/2$. This implies that
$$\dot{w}_\varepsilon(t_1)<w_{\varepsilon}(t_1)-\frac{4w_{\varepsilon}(t_1)}{(\eta(t_1)+w_{\varepsilon}(t_1))\eta(t_1)}=\dot{w}^\varepsilon(t_1)$$
in contradiction with the fact that $t_1$ is first time when $w_\varepsilon=w^\varepsilon$. This proves the claim.

Now, as we mentioned, $y_\varepsilon<y^\varepsilon$ will hold before $w_\varepsilon=\eta$ happens. By the claim, it will never happen. This gives us that $\displaystyle\lim_{t\rightarrow+\infty}y_{\varepsilon}(t)=0$, and then $\displaystyle\lim_{t\rightarrow 1}Y_{\varepsilon}(t)=0$, so $(x,\varepsilon)\in J_1$ for sufficient small $\varepsilon$. And the estimation of $\varepsilon$ depends continuously on $\eta$, and $\eta$ depends continuously on $x$, this finishes the proof.
\end{proof}
Actually, lemma \ref{LPIR} proves that, except for $x_2$ which corresponds to the maximal captured solution, there is no limit point in $\mathbb{R}$ if the driving function satisfies this condition, Thus we only need to prove that there is no limit point in $\mathbb{H}$ as well.
There are many driving functions which satisfy the hypothesis of lemma \ref{LPIR} which may lead the existence of $\displaystyle\lim_{z\rightarrow\lambda(1)}g^{-1}_1(z)$, \ref{THM3} is just an example. Now we only need to prove that the driving function in theorem \ref{THM3} satisfies the hypothesis of the lemma \ref{LPIR}.

\begin{lemma}\label{last}
If $\lambda$, the driving function, satisfies $\forall t\geq 0,a<\frac{\lambda(1)-\lambda(t)}{\sqrt{1-t}}<b$ and  $a>\max\{4,(b+\sqrt{b^2-16})/2\}$, then there exists infinitely many captured solution at time $1$, s.t. the corresponding $\eta$ satisfies $\displaystyle\varliminf_{t\rightarrow+\infty}\eta(t)>0$.
\end{lemma}  
\begin{proof}
We consider the solution $x$ of (\ref{HRLE}) with initial value $x_0$ satisfying $(b+\sqrt{b^2-16})/2<x_0<a$. It is easy to check that the solution is decreasing when $x(t)>(b+\sqrt{b^2-16})/2$ and increasing when $x(t)<(a+\sqrt{a^2-16})/2$. So the solution will lie between $(b+\sqrt{b^2-16})/2$ and $(a+\sqrt{a^2-16})/2$ after a positive time, which implies that $x$ is a captured solution and the corresponding $\eta>a-(b+\sqrt{b^2-16})/2$ after that time. This finishes the proof of this lemma.
\end{proof}

Now theorem \ref{THM3} can be proved.
\begin{proof}[Proof of theorem \ref{THM3}]
At first, it is easy to see that that the condition in theorem \ref{THM3} is equivalent to $a>\max\{5,(b+\sqrt{b^2-16})/2\}$, after some times,  the driving function will satisfy the condition of lemma \ref{last}. And the driving function also satisfies the hypothesis of lemma \ref{LPIR}, hence there is an unique limit point $x_2$ in $\mathbb{R}$, and the captured solution which corresponds to $x_2$ is the minimal captured solution. So we only need to show that there is also no limit points in $\mathbb{H}$.

If there is a limit point $z$ in $\mathbb{H}$, we can get the corresponding  solution of the Loewner equation after time change. Denote $x$ and $y$ to be its real and imaginary part. Same as lemma \ref{LPIR}, we can define $\eta$ and $\tilde{x}$, and set $w=x-\tilde{x}$. In all the captured solutions of $\xi$, $\tilde{x}_2$ is the corresponding solution of $x_2$,and $\tilde{x}_2$ is the minimal captured solution. At first, since $y$ is a vanishing solution which driven by $|\xi-x|$, from lemma \ref{COM} and its remark, $\xi-\tilde{x}_2>2$, hence $x<\tilde{x}_2$ can not always happen. And all the captured solutions except $\tilde{x}_2$ are converge to $\tilde{x}$, hence $x$ will be greater than $\tilde{x}$ at some time, and after that, $x>\tilde{x}$ always hold.

We can  see that $\eta$ have not only a positive lower bound but also an upper bound, and this upper bound decreasing to 0 as $a$ increasing to $+\infty$. It is easy to check that, when $a>5$, this upper bound is less than $2$. Hence we consider the imaginary Loewner equation of $y$,  since $y$ is a vanishing solution which driven by $\eta-w$, there are only two cases: the first case is that after a sufficient large time, $y$ has a positive lower bound, and this lower bound increasing to $2$ as $a$ increasing to $+\infty$. Like the last lemma, we consider the equation of $w$, for a solution $w$ with initial value $0$,
\begin{align*}
 \dot{w}&=w-\frac{4w}{(\eta-w)\eta}+\frac{4}{(\eta-w)^2+y^2}\frac{y^2}{\eta-w}\\
&=w+\frac{4}{\eta-w}\left(\frac{y^2}{(\eta-w)^2+y^2}-\frac{w}{\eta}\right)
\end{align*}

if $w<\eta$ and $\eta\leq 2y$, it is easy to check that the last term is positive. This upper bound of $\eta$ is sufficiently small when $a$ is sufficiently large, but the lower bound of $y$ tends to $2$, hence $\eta>2y$ will not happen, actually, $a\leq 5$ can make sure that $w$ is increasing and $\dot{w}>w$. Then at some finite time, $w=\eta$ will happen, that is $x=\xi$. By the real Loewner after the time change, we obtain that when $x>\xi$, $x$ is still increasing  and $\dot{x}>x$. Thus $x$ will increasing to $+\infty$ exponentially, this makes a contradiction to the fact that $y$ is vanishing.

In the second case, $y$ decreasing to $0$. Because that $\eta$ has an upper bound less than $2$, $y$ will decreasing to $0$ exponentially. Now we consider the small circles of lemma \ref{LPIR}, after a finite time, $(x,y)$ will be in one of them. But the points of the circles are the inner point of $J_1$, which gives us a contradiction.

\begin{remark}
In this proof, $5$ can be improved but can not reach $4$. But we believe that theorem \ref{THM3} is also true when $a>4$, we made a wrong proof of $a>4$ in the previous version. Thanks Joan Lind for pointing out the gap of that proof.
\end{remark}

\end{proof}
\bibliography{ref}
\bibliographystyle{plain}
\end{document}